\documentclass[reqno,12pt]{amsart}

\usepackage{geometry,amsmath,amsfonts,amssymb,amsthm,amscd,mathrsfs,graphicx,enumerate,multicol,wrapfig,subfigure,hyperref,stmaryrd, accents}
\usepackage[all]{xy}
\usepackage{setspace}
\usepackage[T1]{fontenc}

\numberwithin{equation}{subsection}

\newcommand{\ra}{\rightarrow}

\newcommand{\lra}{\longrightarrow}
\newcommand{\8}{\infty}
\newcommand{\p}{\prime}

\newcommand{\al}{\alpha}

\newcommand{\om}{\omega}

\newcommand{\vp}{\varphi}
\newcommand{\lam}{\lambda}

\newcommand{\q}{\theta}

\newcommand{\be}{\beta}
\newcommand{\dt}{\delta}

\newcommand{\Zbb}{\mathbb{Z}}
\newcommand{\Pbb}{\mathbb{P}}
\newcommand{\Cbb}{\mathbb{C}}

\theoremstyle{plain} 

\newtheorem{THM}{Theorem}[section]
\newtheorem{DEF}[THM]{Definition}
\newtheorem{EX}[THM]{Example}

\newtheorem{PROP}[THM]{Proposition}
\newtheorem{LEM}[THM]{Lemma}

\newtheorem{REM}[THM]{Remark}

\newtheorem*{LIFT}{Lifting Property}

\newcommand{\bt}{\bullet}

\newcommand{\id}{\mathrm{id}}

\newcommand{\Vc}{\mathcal{V}}

\newcommand{\Oc}{\mathcal{O}}
\newcommand{\red}{\mathrm{red}}
\newcommand{\Hom}{\mathrm{Hom}}

\newcommand{\Jc}{\mathcal{J}}

\newcommand{\Gc}{\mathcal{G}}

\newcommand{\Fc}{\mathcal{F}}

\newcommand{\Abb}{\mathbb{A}}

\newcommand{\Xfr}{\mathfrak{X}}
\newcommand{\Ec}{\mathcal{E}}

\newcommand{\Ufr}{\mathfrak{U}}

\newcommand{\Zc}{\mathcal{Z}}

\newcommand{\Yfr}{\mathfrak{Y}}
\newcommand{\Ic}{\mathcal{I}}
\newcommand{\Qcl}{\mathcal{Q}}

\usepackage{xcolor}
\definecolor{airforceblue}{rgb}{0.36, 0.54, 0.66}
\definecolor{burgundy}{rgb}{0.5, 0.0, 0.13}
\definecolor{majorelleblue}{rgb}{0.38, 0.31, 0.86}
\definecolor{darkblue}{rgb}{0.0, 0.0, 0.55}

\hypersetup{urlcolor=burgundy, linkcolor=burgundy,
citecolor=darkblue,
colorlinks=true}

\title{
Embeddings of Complex Supermanifolds 
\\}
\author{\small Kowshik Bettadapura}
\date{}

\begin{document}
\maketitle

\begin{abstract} 
In this article we present a study of embeddings of complex supermanifolds. We are broadly guided by the question: when will a submanifold of a split supermanifold itself be split? As an application of our study, we will address this question for certain superspace embeddings over rational normal curves. 
\end{abstract}

\setcounter{tocdepth}{1}
\tableofcontents

\onehalfspacing

\section*{Introduction}

\noindent
One of the central questions in complex supermanifold theory, pertaining to classification, is the splitting question: \emph{given a complex supermanifold, is it split or non-split?} To address this question it suffices to show that any obstruction to the existence of a splitting will vanish. These obstructions are certain cohomology classes concentrated in degree one and can be directly related to glueing data. Hence they can form a basis on which to classify complex supermanifolds. 
\\\\
From treatments of the splitting question, as in \cite{YMAN}, one learns that representatives for obstruction classes can be obtained by suitably differentiating transition functions. Transition functions can however be laborious to derive and depend on a host of extraneous data, such as a covering and chart maps. To get further insight on the splitting question, it would be desirable to find alternative methods to describe these obstructions. The method promoted in this article is the following: if we have a supermanifold $\Yfr$ and want to show it is split, embed it into a \emph{split} supermanifold $\Xfr$ and try to inherit a splitting of $\Yfr$ from the given splitting of $\Xfr$.
\\\\
The central ideas in this article can find their inspiration in the work of Donagi and Witten in \cite{DW1}, where they prove non-splitness of the moduli space of super Riemann surfaces (SRS). There, it is derived the following correspondence of obstruction classes: those of a given supermanifold with those of its submanifolds. By viewing deformations of an SRS as supermanifolds embedded in the corresponding (punctured) moduli space, the obstruction classes of the moduli space can be related to those of the deformation---the latter being significantly easier to describe. We note that Donagi and Witten were concerned with the splitting problem of the ambient supermanifold by reference to its submanifolds. In this article we consider, in a sense, a reverse picture. We are concerned instead with the splitting problem of submanifolds by reference to the ambient supermanifold, which we are at liberty to assume is split.  
\\\\
This article culminates in Theorem \ref{orohfuirgfuirhf98h49} which addresses the splitting question for embeddings over rational normal curves. In degree two this embedding is the superspace quadric, which is a classical example of a non-split supermanifold. It was originally described by Green in \cite{GREEN} and considered in more detail by Onishchik and Bunegina in \cite{ONIBUG}. Witten in \cite[p. 8]{WITTNS} gives a heuristic argument as to why the superspace quadric is non-split. Our deduction of non-splitness of this quadric in Theorem \ref{orohfuirgfuirhf98h49}\emph{(i)}, which is based ultimately on Lemma \ref{krknjviurvburnvien} and \eqref{rhf4gf78hf983jf389fh39hf3}, follows in spirit the argument given by Witten. 

\subsection*{Remark} In a subsequent paper \cite{BETTQUAD}, we argue that non-splitness of quadrics is a generic feature.

\subsection*{Article Summary and Main Results}
We begin with some preliminary theory and establish definitions relevant for our purposes in this article. Generally speaking, we look to describe obstruction classes of submanifolds $\Yfr$ of a given, complex supermanifold $\Xfr$ with a prescribed splitting type. Depending on the splitting type of $\Xfr$ relative to that of $\Yfr$, we show that the obstruction classes of $\Yfr$ can be reduced to computing global sections of certain sheaves. This is the content of Theorem \ref{h89h890309i390}. We show how these sheaves are related to certain, twisted conormal sheaves in Theorem \ref{rhf7hf98fh894f03j903j309}. Subsequently, we propose a correspondence between ideal sheaves and submanifolds of supermanifolds in \eqref{008h489fh94f94fj4}. This is clarified in Theorem \ref{lfvrnvnrnvrnvknvlke} for `even' embeddings. We conclude our study of embeddings with Theorem \ref{rf8hg984h9g8gj4j0}, relating generators for ideal sheaves with the obstruction classes to splitting. There are two classes of applications we provide in this article. In Section \ref{rfu4fhfoj4oifi4fmp}, we obtain some general characterisations of (even) embeddings. In Theorem \ref{ldfjnvnjfjdkkskskdkd} we partially address the splitting question for embeddings which motivated this article. This leads to the notion of split embeddings of models, introduced Definition \ref{finioruhf4jfio4jo4joji4}. We apply a classical result to deduce the existence of such embeddings in Example \ref{9008f9gf73f}. In Section \ref{ofprijiovorvorvoirip}, we consider subvarieties of projective superspace. Our main result is Theorem \ref{orohfuirgfuirhf98h49} where we argue that certain superspace embeddings over rational normal curves can non-split only in degree two and are otherwise split. This article concludes with remarks on potential directions for future work. In brief: the ideas in this article should be applicable in studying certain subvarieties of projective and weighted projective superspaces which appear in \cite{SETHI, AGAVAFA} as proposed candidates for mirrors of rigid, K\"ahler manifolds in Landau-Ginzberg sigma models. We address this proposal in a subsequent article \cite{BETTQUAD}.

\subsection*{Conventions}
We work over the field of complex numbers. Pairs $(X, T^*_{X, -})$ are referred to as `models'. They comprise a complex manifold $X$ (usually compact) and a holomorphic vector bundle $T^*_{X, -}$. We view $X$ as a locally ringed space with structure sheaf $\Oc_X$ and $T^*_{X, -}$ as a locally free sheaf of $\Oc_X$-modules. Morphisms are always holomorphic and so pulling back and pushing forward along them will define exact functors on the respective module categories. We refer to \cite{GRAUREM} for the general theory of complex manifolds and sheaves on them. We reference \cite{YMAN, QFAS} for foundational aspects of complex supermanifold theory. The morphisms of supercommutative algebras considered here are required to preserve the parity and so, in this sense, are always `even'.

\section{Preliminaries}

\subsection{Green's Automorphism Groups}
Fix a model $(X, T^*_{X, -})$. Green's sheaf of non-abelian groups $\Gc^{(2)}_{T^*_{X, -}}$ is defined the kernel of the surjective morphism of sheaves of groups $\mathcal Aut~\wedge^\bt T^*_{X, -} \ra \mathcal Aut_{\Oc_X}T^*_{X,-}$. More generally, set $\Jc_{T^*_{X, -}}^{< k} = \oplus_{j=1}^{k-1} \wedge^jT^*_{X, -}$. Green's `higher' sheaves of groups $\Gc^{(k)}_{T^*_{X, -}}$ are defined as the kernel of $\mathcal Aut\wedge^\bt T^*_{X, -}\ra \mathcal Aut~\Jc^{<k}_{T^*_{X, -}}$.

As a set: 
\begin{align}
\Gc^{(k)}_{T^*_{X, -}} 
=
\left\{
\al\in \mathcal Aut\wedge^\bt T_{X, -}
\mid \al(u) - u \in \Jc^k_{T_{X, -}^*}
\right\}.
\label{rhf784gf78h98fj30}
\end{align}
where $\Jc_{T_{X, -}^*} = \oplus_{j\geq 1}\wedge^jT_{X, -}^*$ and $\Jc^k_{T_{X, -}^*} = \oplus_{j\geq k}\wedge^jT_{X, -}^*$.

\begin{DEF}\label{rfh89hf98f0j30jf03jf903}
\emph{Let $(X, T^*_{X,-})$ be a model. For all $k\geq 2$, elements of the \v Cech cohomology set $\mbox{\v H}^1\big(X, \Gc_{T^*_{X, -}}^{(k)}\big)$ will be referred to as \emph{$(k-1)$-split supermanifolds modelled on $(X, T_{X,-}^*)$}. Furthermore:
\begin{enumerate}[$\bt$]
	\item $X$ is referred to as the \emph{reduced space};
	\item $T^*_{X, -}$ is referred to as the \emph{odd, conormal bundle};
	\item A $(k-1)$-split supermanifold $\Xfr$ is said to have of \emph{splitting type $(k-1)$}.
\end{enumerate}}
\end{DEF}

\begin{DEF}\label{dhfh3hf9h398f83jf03j3f33}
\emph{The basepoint in $\mbox{\v H}^1\big(X, \Gc_{T^*_{X, -}}^{(2)}\big)$ will be denoted $e_{(X, T_{X,-}^*)}$. It is referred to as the \emph{split model associated to $(X, T^*_{X, -})$}.\footnote{As a point of clarification, unlike the model $(X, T^*_{X, -})$, the split model is a supermanifold.}}
\end{DEF}

\subsection{Primary Obstructions}
From \eqref{rhf784gf78h98fj30} it is clear that $\Gc^{(k+1)}_{T^*_{X, -}}\subset \Gc^{(k)}_{T^*_{X, -}}$ for any $k\geq 2$. A fundamental result underpinning the obstruction theory for supermanifolds is the following by Green, proved in \cite{GREEN}.

\begin{LEM}\label{rfngf874f98h0f3j3f3f}
For any model $(X, T_{X,-}^*)$ and each $k\geq 2$,
\begin{enumerate}[(i)]
	\item $\Gc_{T^*_{X, -}}^{(k+1)}$ is a sheaf of normal subgroups of $\Gc^{(k)}_{T^*_{X, -}}$;
	\item the quotient $\Qcl_{T^*_{X,-}}^{(k)}:= \Gc^{(k)}_{T^*_{X, -}}/\Gc^{(k+1)}_{T^*_{X, -}}$ is a sheaf of abelian groups.
\end{enumerate}
\end{LEM}

\noindent
We will take \emph{(i)} as given and prove \emph{(ii)} as it will be referenced in a subsequent section.
\\\\ 
\emph{Proof of Lemma $\ref{rfngf874f98h0f3j3f3f}$(ii)}. 
The following general fact from group theory will be useful: for $N\leq G$ a normal subgroup, the quotient $G/N$ is abelian if and only if $N$ contains the commutator subgroup $[G, G]$. We now claim $\big[\Gc^{(k)}_{T^*_{X, -}}, \Gc^{(k)}_{T^*_{X, -}}\big]\subset \Gc^{(4k)}_{T^*_{X, -}}$. This is easiest to see at the Lie algebra level. Firstly, the Lie algebra $\mathfrak g^{(k)}_{T^*_{X, -}}$ of $\Gc^{(k)}_{T^*_{X, -}}$ can be identified with the sheaf of derivations $\wedge^\bt T^*_{X,-} \ra \wedge^\bt T^*_{X, -}$ of degree $k$, i.e., those derivations sending $\wedge^lT^*_{X,-} \ra \wedge^{l+k}T^*_{X,-}$ for all $l$. This is a nilpotent Lie algebra so therefore the formal exponential map $\mathfrak g^{(k)}_{T^*_{X, -}} \ra \Gc_{T^*_{X, -}}^{(k)}$ will be a bijection of sheaves of sets. Hence for any $\al_1, \al_2\in \Gc^{(2)}_{T^*_{X, -}}$ we can write $\al_1 = e^{x_1}$ and $\al_2 = e^{x_2}$ for $x_1, x_2\in \mathfrak g^{(k)}_{T^*_{X, -}}$. Note $\al_i^{-1} = e^{-x_i}$. From the Campbell-Baker-Hausdorff formula:
\[
\al_1\circ\al_2 = e^{x_1}e^{x_2} = e^{x_1+x_2 + \frac{1}{2}[x_1, x_2] + \ldots}.
\]
Thus the commutator is:
\begin{align}
[\al_1, \al_2]
&=
\al_1\al_2\al_1^{-1}\al_2^{-1}
\notag
\\
&=
e^{x_1}e^{x_2}e^{-x_1}e^{-x_2}
\notag
\\
&= e^{\frac{1}{8} [x_1, [x_2, [x_1, x_2]]]+\ldots}\label{9403983hf98hf83}
\end{align}
It remains to note that the term $[x_1, [x_2, [x_1, x_2]]]$ in \eqref{9403983hf98hf83} is a derivation of $\wedge^\bt T^*_{X, -}$ of degree-$(4k)$. Hence 
$\big[\Gc^{(k)}_{T^*_{X, -}}, \Gc^{(k)}_{T^*_{X, -}}\big]
\subset \Gc^{(4k)}_{T^*_{X, -}}$. 
Since $\Gc^{(k+1)}_{T^*_{X,-}}$ contains $\Gc^{(4k)}_{T^*_{X, -}}$ it will contain the commutator subgroup $[\Gc^{(k)}_{T^*_{X, -}}, \Gc^{(k)}_{T^*_{X, -}}]$. Assuming \emph{(i)} we know that $\Gc^{(k+1)}_{T^*_{X,-}}\subset \Gc^{(k)}_{T^*_{X,-}}$ is normal. Hence the quotient is abelian.
\qed

\begin{DEF}
\emph{For any model $(X, T_{X,-}^*)$, the quotient sheaf $\Qcl^{(k)}_{T^*_{X, -}}$ will be referred to as the \emph{$k$-th obstruction sheaf}. The sheaf cohomology group $H^1\big(X, \Qcl_{T^*_{X, -}}^{(k)}\big)$ will be referred to as the \emph{$k$-th obstruction space}.}
\end{DEF}

\noindent
The short exact sequence of sheaves of groups,
\[
\xymatrix{
\{1\}\ar[r] & \Gc^{(k+1)}_{T^*_{X, -}} \ar[r] & \Gc^{(k)}_{T^*_{X, -}} \ar[r]  & \Qcl_{T^*_{X, -}}^{(k)}\ar[r] & \{1\}
}
\]
induces a long exact sequence (of pointed sets) on \v Cech cohomology containing the piece:
\begin{align}
\cdots\lra
\mbox{\v H}^1\big(X, \Gc^{(k+1)}_{T^*_{X, -}}\big)
\lra 
\mbox{\v H}^1\big(X, \Gc^{(k)}_{T^*_{X, -}}\big)
\stackrel{\om}{\lra}
H^1\big(X, \Qcl_{T^*_{X, -}}^{(k)}\big).
\label{rf74h9f8h80fj039j}
\end{align}
Hence to any $(k-1)$-split supermanifold $\Xfr$ we have a cohomology class $\om(\Xfr)$. 

\begin{DEF}
\emph{The class $\om(\Xfr)$ of a $(k-1)$-split supermanifold $\Xfr$ modelled on $(X, T_{X,-}^*)$ will be referred to as the \emph{primary obstruction} of $\Xfr$.}
\end{DEF}

\noindent
To justify the terminology in the above definition we have the following, which is essentially a restatement of the fact that \eqref{rf74h9f8h80fj039j} is exact.

\begin{LEM}\label{rhf78rhf74hf98fj3f039}
A $(k-1)$-split supermanifold is $k$-split if and only if its primary obstruction vanishes.\qed
\end{LEM}

\noindent
In the interests of classification we give the following definition. It is an adaption of non-splitness as one might traditionally find in the literature.

\begin{DEF}\label{rhf97hf983jf0jf09j30}
\emph{A supermanifold is said to be \emph{non $k$-split} if it is $(k-1)$-split with non vanishing primary obstruction.}
\end{DEF}

\subsection{Classifying Supermanifolds}
For completeness we give a brief summary here of complex supermanifolds as one might traditionally find in the literature, such as in \cite{YMAN, QFAS}. The view of supermanifolds promoted in this article, in Definition \ref{rfh89hf98f0j30jf03jf903}, is as certain classes in a \v Cech cohomology set. More classically, with a fixed model $(X, T_{X, -}^*)$, a supermanifold modelled on  $(X, T_{X, -}^*)$ is defined as locally ringed space $\Xfr = (X, \Oc_\Xfr)$ with $\Oc_\Xfr$ a sheaf of supercommutative algebras, locally isomorphic to $\wedge^\bt T^*_{X, -}$. This means there exists a cover $(U_i)$ of $X$ such that $ \Oc_\Xfr(U_i)\cong \wedge^\bt T^*_{X, -}(U_i)$. The sheaf $\Oc_\Xfr$ is the structure sheaf of $\Xfr$. One says $\Xfr$ is \emph{split} if $\Oc_\Xfr$ is globally isomorphic to $\wedge^\bt T^*_{X, -}$. Since $\Oc_\Xfr$ is supercommutative, it is globally $\Zbb_2$-graded and we write $\Oc_\Xfr = \Oc_\Xfr^+\oplus \Oc_\Xfr^-$. The odd part $\Oc_\Xfr^-$ is an $\Oc_\Xfr^+$-module. It is a submodule of $\Oc_\Xfr$ and the ideal in $\Oc_\Xfr$ it generates is denoted $\Jc_\Xfr$. It satisfies:
\begin{align}
\Oc_\Xfr/ \Jc_\Xfr
= \Oc^+_\Xfr/ \Jc_\Xfr
= \Oc_X
&&
\mbox{and}
&&
\Oc^-_\Xfr/\Jc^2_\Xfr 
=
\Oc_\Xfr^-/ \Jc_\Xfr^2
=
 T^*_{X, -}. 
 \label{rhf4hf984j09fjfj30}
\end{align}
In particular $\Jc_\Xfr/\Jc_\Xfr^2$ is locally free. Note $\Oc_\Xfr/\Jc_\Xfr^2 = \Oc_X\oplus T^*_{X, -}$.\footnote{If $\Xfr$ is split then $\Jc_\Xfr\cong \oplus_{j>0}\wedge^jT^*_{X, -}$.} There is a useful criterion for splitting which we can obtain directly from the descriptions in \eqref{rhf4hf984j09fjfj30}. Consider rewriting these descriptions as exact sequences:
\begin{align}
0 \ra \Jc_\Xfr\ra \Oc_\Xfr \ra \Oc_X\ra0
&&\mbox{and}&&
0\ra \Jc^2_{T^*_{X, -}} \ra \Jc_{T^*_{X, -}} \ra T^*_{X, -}\ra0.
\label{fniooiuheoifhoifoi3joj3o}
\end{align}
We have:

\begin{LEM}\label{rbf8gf87h4f98h389fj0j30}
If the sequences in \eqref{fniooiuheoifhoifoi3joj3o} are both split exact, then $\Xfr$ is split. 
\end{LEM}

\begin{proof}
This is a well-known characterisation of splitting for supermanifolds. For completeness we provide a proof in Appendix \ref{rf67f47gf87hf983hf03j0}. 
\end{proof}

\noindent
In the paper by Green in \cite{GREEN} it is shown that $\mbox{\v H}^1\big(X, \Gc_{T^*_{X, -}}^{(2)}\big)$ classifies complex supermanifolds (as the locally ringed spaces described above) up to an appropriate equivalence. Up to isomorphism, supermanifolds are classified by their image in $\mbox{\v H}^1\big(X, \mathcal Aut\wedge^\bt T^*_{X, -}\big)$ under the natural map induced on cohomology by the inclusion $\Gc_{T^*_{X, -}}^{(2)}\subset  \mathcal Aut\wedge^\bt T^*_{X, -}$. A supermanifold is split if and only if its image in $\mbox{\v H}^1\big(X, \mathcal Aut\wedge^\bt T^*_{X, -}\big)$ coincides with the basepoint. Otherwise, it is non-split.
\\\\
It is generally quite difficult to find classes which obstruct the existence of a splitting. As such we consider instead the notion of `$(k-1)$-splitting' as in Definition \ref{rfh89hf98f0j30jf03jf903}. In the terminology of Definition \ref{rhf97hf983jf0jf09j30}, we have: any non $2$-split supermanifold is in fact non-split. This is a classical result and a proof is given in the appendix in \cite{BETTHIGHOBS}. The analogous statement for non $k$-split supermanifolds for $k> 2$ does not necessarily hold. We refer again to \cite{BETTHIGHOBS} for further discussions on this point.

\section{Embeddings}

\subsection{Definitions}
Consider models $(Y, T^*_{Y,-})$ and $(X, T_{X,-}^*)$. Suppose we have an holomorphic embedding of spaces $i: Y\subset X$ and a surjection of sheaves $f^\sharp: T^*_{X, -} \ra  i_*T_{Y,-}^*$. If these maps exist we will say there exists a holomorphic embedding of models $f : (Y, T_{Y,-}^*) \subset (X, T_{X,-}^*)$, where $f = (i, f^\sharp)$. Since the taking the exterior algebra is right-exact, it follows that $\wedge^\bt T^*_{X, -} \ra \wedge^\bt f_* T_{Y, -}^*$ is surjective. We denote the kernel by $\Ic_{T^*_{Y, -};T^*_{X, -}}$. Now, not every automorphism of $\wedge^\bt T^*_{X, -}$ will induce an automorphism of $\wedge^\bt T^*_{Y, -}$. Only those automorphisms preserving $\Ic_{T^*_{Y, -};T^*_{X, -}}$. Let $\mathcal Aut_{T^*_{Y,-}; T^*_{X, -}}$ denote the subgroup of such automorphisms and set
\[
\Gc^{(k)}_{T^*_{Y, -};T_{X, -}^*}
:=
\Gc^{(k)}_{T_{X, -}^*} \cap \mathcal Aut_{T^*_{Y,-}; T^*_{X, -}}.
\]
We have natural homomorphisms of sheaves of groups:
\begin{align}
\xymatrix{
\ar[d]_r \Gc^{(k)}_{T^*_{Y, -};T_{X, -}^*} \ar[rr]^u & & \Gc^{(k)}_{T_{X, -}^*}
\\
\Gc^{(k)}_{T^*_{Y, -}}& & 
}
\label{gf874gf7h98f3j0}
\end{align}
where $u$ is the inclusion and $r$ is the restriction of a group element to the submanifold $Y\subset X$. The maps in \eqref{gf874gf7h98f3j0} induce a similar picture on \v Cech cohomology:
\begin{align}
\xymatrix{
\ar[d]_{r_*} 
\mbox{\v H}^1\big(X, \Gc^{(k)}_{T^*_{Y, -};T_{X, -}^*}\big) \ar[rr]^{u_*} & & \mbox{\v H}^1\big(X,\Gc^{(k)}_{T_{X, -}^*}\big)
\\
\mbox{\v H}^1\big(Y, \Gc^{(k)}_{T^*_{Y, -}}\big)& & 
}
\label{rhfi4gf74hf8939fj3}
\end{align}
Hence to any element in $\mbox{\v H}^1\big(X, \Gc^{(k)}_{T^*_{Y, -};T_{X, -}^*}\big)$ we can assign $(k-1)$-split supermanifolds modelled on $(Y, T_{Y,-}^*)$ and $(X, T_{X,-}^*)$ respectively.

\begin{DEF}\label{fj89hf98f93jf9jf39}
\emph{Let $\Yfr$ and $\Xfr$ be $(k-1)$-split supermanifolds modelled on $(Y, T_{Y,-}^*)$ and $(X, T_{X,-}^*)$ respectively. Fix a holomorphic embedding $f: (Y, T_{Y,-}^*)\subset(X, T_{X,-}^*)$ of models. We say there exists a \emph{holomorphic embedding} of $\Yfr$ in $\Xfr$ \emph{over} $f$, or simply an \emph{embedding} (with $f$ understood) if and only if there exists some $\mathcal Z\in \mbox{\v H}^1\big(X, \Gc^{(k)}_{T^*_{Y, -};T_{X, -}^*}\big)$ such that 
\begin{align*}
r_*(\Zc) = \Yfr  &&\mbox{and} && u_*(\Zc) = \Xfr.
\end{align*}
We denote an embedding by $\big(\Zc: \Yfr\subset \Xfr\big)$.}
\end{DEF}

\begin{DEF}\label{rufuirbiuehfuheoifeojfoe}
\emph{Let $f: (Y, T^*_{Y, -})\subset (X, T^*_{X, -})$ be an embedding of models. Elements of $\mbox{\v H}^1\big(X, \Gc^{(k)}_{T^*_{Y, -};T_{X, -}^*}\big)$ are referred to as \emph{$(k-1)$-split embeddings over $f$}, or simply \emph{$(k-1)$-split embeddings} with $f$ understood. Following Definition \ref{rfh89hf98f0j30jf03jf903}, we will refer to elements of $\mbox{\v H}^1\big(X, \Gc^{(k)}_{T^*_{Y, -};T_{X, -}^*}\big)$ as having \emph{splitting type} $(k-1$).} 
\end{DEF}

\noindent
From the definition of an embedding it is clear that if $\Yfr$ and $\Xfr$ are supermanifolds and $\big(\Zc: \Yfr\subset \Xfr\big)$ is a $(k-1)$-split embedding, then $\Yfr$ and $\Xfr$ must both be $(k-1)$-split.

\begin{REM}\label{rfg784gf87h9f39j30}
\emph{Note that the diagram in \eqref{rhfi4gf74hf8939fj3} was constructed only from the data of an embedding $f: (Y,T_{Y,-}^*)\subset (X,T_{X,-}^*)$. Hence if there exists such an embedding $f$, there will exist an embedding of split, $(k-1)$-split (i.e., $k$-split) supermanifolds for all $k$. This is simply because $u_*$ and $r_*$ in \eqref{rhfi4gf74hf8939fj3} are maps of pointed sets and so map basepoints to basepoints. Hence, using the notation in Definition \ref{dhfh3hf9h398f83jf03j3f33}, we see that an embedding of models $(Y,T_{Y,-}^*)\subset (X,T_{X,-}^*)$ gives an embedding of respective split models $e_{(Y, T_{Y,-}^*)}\subset e_{(X, T_{X,-}^*)}$.}
\end{REM}

\subsection{Splitting Types}
For any $k^\p\geq k$ there exists a natural map $\Gc^{(k^\p)}_{T_{Y, -}^*;T_{X, -}^*} \ra\Gc^{(k)}_{T_{Y, -}^*;T_{X, -}^*}$ induced from the inclusion $\Gc^{(k^\p)}_{T^*_{X, -}} \subset \Gc^{(k)}_{T^*_{X, -}}$. This leads to the following commutative diagram:
\begin{align}
\xymatrix{
\ar[d] \ar[dr] \Gc^{(k^\p)}_{T_{Y, -}^*;T_{X, -}^*} \ar[r] & \Gc^{(k^\p)}_{T_{X, -}^*}\ar[dr]
\\
\ar[dr] \Gc^{(k^\p)}_{T_{Y, -}^*} &\Gc^{(k)}_{T_{Y, -}^*;T_{X, -}^*}\ar[d] \ar[r] & \Gc^{(k)}_{T_{X, -}^*}
\\
& \Gc^{(k)}_{T_{Y, -}^*}
}\label{rhf784f794hf8jf03j}
\end{align}
And hence on cohomology:
\begin{align*}
\xymatrix{
\ar[d]\mbox{\v H}^1\big(X, \Gc^{(k^\p)}_{T^*_{Y, -};T_{X, -}^*}\big) \ar[dr]\ar[r]  &\mbox{\v H}^1\big(X,\Gc^{(k^\p)}_{T_{X, -}^*}\big)\ar[dr]
\\
\mbox{\v H}^1\big(Y, \Gc^{(k^\p)}_{T^*_{Y, -}}\big)\ar[dr] 
&\ar[d]\mbox{\v H}^1\big(X, \Gc^{(k)}_{T^*_{Y, -};T_{X, -}^*}\big) \ar[r]  & \mbox{\v H}^1\big(X,\Gc^{(k)}_{T_{X, -}^*}\big)
\\
&\mbox{\v H}^1\big(Y, \Gc^{(k)}_{T^*_{Y, -}}\big)& 
}
\end{align*}
The above diagram shows that it is possible for there to exist an embedding $\big(\Zc: \Yfr\subset \Xfr\big)$ with $\Yfr$ and $\Xfr$ having different splitting types. This leads to the  following definition.

\begin{DEF}
\emph{The \emph{total splitting type} of an embedding $\big(\Zc:\Yfr\subset \Xfr\big)$ is the triple of integers $(k; k^\p, k^{\p\p})$, each greater than $1$, and where:
\begin{enumerate}[(i)]
	\item $(k - 1)$ is the splitting type of $\Zc$;
	\item $(k^\p-1)$ is the splitting type of $\Yfr$ and;  
	\item $(k^{\p\p}-1)$ is the splitting type of $\Xfr$.
\end{enumerate}}
\end{DEF}

\begin{REM}
\emph{In this article we will be interested in embeddings of total splitting type $(k; k, k+1)$. Such embeddings subsume, for instance, submanifolds of split supermanifolds, which is the setting for our intended applications.}
\end{REM}

\section{Obstructions}

\subsection{Normality}
Central to the classical obstruction theory for supermanifolds is Green's normality result in Lemma \ref{rfngf874f98h0f3j3f3f}. We will prove an analogous result for the sheaves $\Gc^{(k)}_{T^*_{Y, -}; T^*_{X, -}}$.

\begin{LEM}\label{hf94f9h498fj03jf3}
For each $k\geq 2$ there exists an embedding $\Gc^{(k+1)}_{T^*_{Y, -}; T^*_{X, -}} \subset \Gc^{(k)}_{T^*_{Y, -}; T^*_{X, -}}$ realising $\Gc^{(k+1)}_{T^*_{Y, -}; T^*_{X, -}}$ as a sheaf of normal subgroups of $\Gc^{(k)}_{T^*_{Y, -}; T^*_{X, -}}$. 
\end{LEM}

\begin{proof}
We will use the following classical result about groups and normal subgroups:
\begin{enumerate}[($\star$)]
	\item let $G$ be a group; $H< G$ a subgroup and $N\leq G$ a normal subgroup. Then $H\cap N$ is a normal subgroup of $H$.
\end{enumerate}
Green's lemma states $\Gc^{(k+1)}_{T^*_{X, -}}$ is a normal subgroup of $\Gc^{(k)}_{T^*_{X, -}}$. Now, by definition $\Gc^{(k)}_{T^*_{Y, -}; T^*_{X, -}} = \Gc^{(k)}_{T^*_{X, -}} \cap \mathcal Aut_{T^*_{Y,-}}\wedge^\bt T_{X, -}^*$. It is a subgroup of $\Gc^{(k)}_{T^*_{X, -}}$ and so by $(\star)$ above $\Gc^{(k+1)}_{T^*_{X, -}}\cap \Gc^{(k)}_{T^*_{Y, -}; T^*_{X, -}}$ will be a normal subgroup of $\Gc^{(k)}_{T^*_{Y, -}; T^*_{X, -}}$. Now note that 
\begin{align*}
\Gc_{T^*_{X, -}}^{(k+1)}\cap \Gc^{(k)}_{T^*_{Y, -}; T^*_{X, -}} 
&=
\left(\Gc_{T^*_{X, -}}^{(k+1)}\cap \Gc^{(k)}_{T^*_{X, -}} \right)\cap \mathcal Aut_{T^*_{Y,-}}\wedge^\bt T_{X, -}^*
\\
&=
\Gc_{T^*_{X, -}}^{(k+1)}\cap \mathcal Aut_{T^*_{Y,-}}\wedge^\bt T_{X, -}^*
\\
&=
\Gc^{(k+1)}_{T^*_{Y, -}; T^*_{X, -}} 
\end{align*}
The lemma now follows.
\end{proof}

\begin{REM}
\emph{As remarked in \cite{DW1}, the sheaf of groups $\Gc^{(k)}_{T^*_{Y, -}; T^*_{X, -}}$  need \emph{not} be a sheaf of normal subgroups of $\Gc^{(k)}_{T^*_{X, -}}$.}
\end{REM}

\subsection{The Obstruction Sheaves}
We will denote the quotient $\Gc^{(k)}_{T^*_{Y, -}; T^*_{X, -}}/\Gc^{(k+1)}_{T^*_{Y, -}; T^*_{X, -}}$ by the sheaf $\Qcl_{T^*_{Y, -};T^*_{X, -}}^{(k)}$. It is a sheaf of groups by Lemma \ref{hf94f9h498fj03jf3} above. Like the obstruction sheaves $\Qcl_{T^*_{X, -}}^{(k)}$ we have:

\begin{LEM} 
$\Qcl_{T^*_{Y, -};T^*_{X, -}}^{(k)}$ is a sheaf of abelian groups.
\end{LEM}

\begin{proof}
This follows from the same argument as in Lemma \ref{rfngf874f98h0f3j3f3f}\emph{(ii)}.
\end{proof}

\begin{DEF}
\emph{Let $(Y, T_{Y,-}^*)\subset (X, T_{X,-}^*)$ be an embedding of models. The abelian sheaves $\Qcl_{T^*_{Y, -};T^*_{X, -}}^{(k)}$ associated to this embedding will be referred to as the \emph{$k$-th obstruction sheaves for the embedding}.}
\end{DEF}

\noindent
From commutativity of \eqref{rhf784f794hf8jf03j} we see that there will be induced the following maps on the obstruction sheaves:
\begin{align}
\xymatrix{
\ar[d]\Qcl_{T^*_{Y, -};T^*_{X, -}}^{(k)}\ar[r] & \Qcl_{T^*_{X, -}}^{(k)}
\\
\Qcl_{T^*_{Y, -}}^{(k)}
}
\label{rfj94fj093jf093f3fkf3}
\end{align}
Hence for each $k$ we have on cohomology:
\begin{align}
\xymatrix{
\ar[dr]\ar[d]\mbox{\v H}^1\big(X, \Gc^{(k)}_{T^*_{Y, -}; T^*_{X, -}}\big)
\ar[r] & \mbox{\v H}^1\big(X, \Gc^{(k)}_{T^*_{X, -}}\big)\ar[dr]
\\
\ar[dr]\mbox{\v H}^1\big(Y, \Gc^{(k)}_{T^*_{Y, -}}\big) &\ar[d] H^1\big(X, \Qcl_{T^*_{Y, -};T^*_{X, -}}^{(k)}\big)\ar[r] & H^1\big(X, \Qcl_{T^*_{X, -}}^{(k)}\big)
\\
&H^1\big(Y, \Qcl_{T^*_{Y, -}}^{(k)}\big) & 
}
\label{rf784f7h4f8jf09jf30}
\end{align}
Just like $(k-1)$-split supermanifolds we have the following definition.

\begin{DEF}
\emph{Let $\big(\Zc: \Yfr\subset \Xfr\big)$ be a $(k-1)$-split embedding. The image of $\big(\Zc: \Yfr\subset \Xfr\big)$ in $H^1\big(X, \Qcl_{T^*_{Y, -};T^*_{X, -}}^{(k)}\big)$ under the map in \eqref{rf784f7h4f8jf09jf30} will be referred to as the \emph{primary obstruction} of the embedding $\big(\Zc: \Yfr\subset \Xfr\big)$}
\end{DEF}

\noindent
By Lemma \ref{hf94f9h498fj03jf3} we are guaranteed the following result, analogous to Lemma \ref{rhf78rhf74hf98fj3f039} for supermanifolds.

\begin{LEM}\label{rbcigihf9hf93hoijeo}
A $(k-1)$-split embedding is $k$-split if and only if its primary obstruction vanishes.\qed
\end{LEM}

\noindent
The relation of the primary obstructions of embeddings to those of supermanifolds can be readily deduced from commutativity of \eqref{rf784f7h4f8jf09jf30}. 

\begin{PROP}\label{rfb4fhf98f80j09fj930j}
Let $\big(\Zc: \Yfr\subset \Xfr\big)$ be a $(k-1)$-split embedding. Then under the maps in \eqref{rf784f7h4f8jf09jf30}, the primary obstruction $\om(\Zc)$ will map to $\om(\Yfr)$ and $\om(\Xfr)$ respectively, i.e., we have: 
\[
\xymatrix{
\ar@{|->}[d]\om(\Zc) \ar@{|->}[r] & \om(\Xfr)
\\
\om(\Yfr)
}
\]\qed
\end{PROP}

\subsection{Obstructions to Existence}
Based on the primary obstructions of $\Xfr$, it is possible to deduce whether there will exist submanifolds $\Yfr\subset \Xfr$. The starting point if the following. 

\begin{LEM}\label{rfuefhfiuefeiojoeidep}
Let $(Y, T_{Y,-}^*)\subset (X, T_{X,-}^*)$. Then for each $k\geq2$, $\Qcl_{T^*_{Y, -};T^*_{X, -}}^{(k)}$ is a subsheaf of $\Qcl_{T^*_{X, -}}^{(k)}$. 
\end{LEM}

\begin{proof}
By commutativity of \eqref{rhf784f794hf8jf03j} we have induced a map $\iota:\Qcl_{T^*_{Y, -};T^*_{X, -}}^{(k)}\ra \Qcl_{T^*_{X, -}}^{(k)}$ giving rise to the following morphism of short exact sequences of sheaves of groups:
\begin{align}
\xymatrix{
\ar[d] \Gc^{(k+1)}_{T^*_{Y, -};T^*_{X, -}} \ar[r]  & \Gc^{(k)}_{T^*_{Y, -};T^*_{X, -}}\ar[d]  \ar[r]  &\Qcl_{T^*_{Y, -};T^*_{X, -}}^{(k)}\ar@{-->}[d]^\iota
\\
\Gc^{(k+1)}_{T^*_{X, -}} \ar[r] & \Gc^{(k)}_{T^*_{X, -}} \ar[r]  & \Qcl_{T^*_{X, -}}^{(k)}
}
\label{rf74hf9hf8j30fj093}
\end{align}
The solid, vertical arrows are injective. We wish to show that the dashed arrow $\iota$ is also injective. To see this, observe that $\ker \iota$ can be identified with a subgroup of the image of $\Gc^{(k)}_{T^*_{Y, -};T^*_{X, -}} \cap \Gc^{(k+1)}_{T^*_{X, -}}$ in $\Gc^{(k)}_{T^*_{X, -}}$. This follows from short-exactness of the rows in \eqref{rf74hf9hf8j30fj093}. Now note that this intersection is precisely $\Gc^{(k+1)}_{T^*_{Y, -};T^*_{X, -}}$ by definition. Hence $\ker \iota\subset \Gc^{(k+1)}_{T^*_{Y, -};T^*_{X, -}}$. But $\Qcl_{T^*_{Y, -};T^*_{X, -}}^{(k)} = \Gc^{(k)}_{T^*_{Y, -};T^*_{X, -}}/ \Gc^{(k+1)}_{T^*_{Y, -};T^*_{X, -}}$ which means we must have $\ker \iota = (0)$ and so $\iota$ is injective. 
\end{proof}

\noindent
Let $\mathcal R^{(k)}_{T^*_{Y, -}, T^*_{X, -}}$ denote the quotient $\Qcl_{T^*_{X, -}}^{(k)}/\Qcl_{T^*_{Y, -};T^*_{X, -}}^{(k)}$. Then we will have a long exact sequence on cohomology containing the following exact piece:
\[
\cdots
\lra 
H^1\big(X, \Qcl_{T^*_{Y, -};T^*_{X, -}}^{(k)}\big) 
\stackrel{\iota_*}{\lra}
H^1\big(X, \Qcl_{T^*_{X, -}}^{(k)}\big)
\stackrel{\be_{Y;X}}{\lra} 
H^1\big(X, \mathcal R^{(k)}_{T^*_{Y, -}, T^*_{X, -}}\big)
\lra \cdots
\]
Now, the map $H^1\big(X, \Qcl_{T^*_{Y, -};T^*_{X, -}}^{(k)}\big) \ra H^1\big(X, \Qcl_{T^*_{X, -}}^{(k)}\big)$ from \eqref{rf784f7h4f8jf09jf30} is induced from the embedding $\iota: \Qcl_{T^*_{Y, -};T^*_{X, -}}^{(k)}\subset \Qcl_{T^*_{X, -}}^{(k)}$. Hence from Proposition \ref{rfb4fhf98f80j09fj930j}, if there exists an embedding of supermanifolds $\big(\Zc: \Yfr\subset \Xfr\big)$, then $\iota_*\om(\Zc) = \om(\Xfr)$. This leads to the following obstruction-to-existence result.

\begin{THM}\label{rbfy4gf78h47fh389f3}
Let $(Y, T_{Y,-}^*)\subset (X,T_{X,-}^*)$ be an embedding of models. For any $k\geq 2$, if $\Xfr$ is a $(k-1)$-split supermanifold with primary obstruction $\om(\Xfr)$ such that $\be_{Y;X}(\om_\Xfr)\neq0$, then there will not exist any $(k-1)$-split submanifold of $\Xfr$ modelled on $(Y, T_{Y,-}^*)$. \qed
\end{THM}

\subsection{A Correspondence of Obstructions}
In this article we only consider holomorphic embeddings of models $(Y, T^*_{Y, -})\subset (X, T^*_{X, -})$. This means the embedding of underlying spaces $i: Y\subset X$ is holomorphic. As such the restriction functor $r = i^*$ from sheaves on $X$ to sheaves on $Y$ is exact (see e.g., \cite[p. 20]{GRAUREM}). From Lemma \ref{rfuefhfiuefeiojoeidep} we therefore obtain the following commutative diagram,
\begin{align}
\xymatrix{
0\ar[r] & \Qcl^{(k)}_{T^*_{Y, -}; T^*_{X, -}}\ar[d]^r \ar[r] &  \Qcl^{(k)}_{T^*_{X, -}}\ar[r] \ar[d]^r&\mathcal R^{(k)}_{T^*_{Y, -}; T^*_{X, -}}\ar[r]\ar[d]^r & 0
\\
0\ar[r] & \Qcl^{(k)}_{T^*_{Y, -}} \ar[r]^i &  \Qcl^{(k)}_{T^*_{X, -}}|_Y\ar[r] &\mathcal R^{(k)}_{T^*_{Y, -}; T^*_{X, -}}|_Y\ar[r] & 0
}
\label{rrtgiug484848hirgjrioiuii}
\end{align}
This diagram translates to a commutative diagram on cohomology. Upon combining it with \eqref{rf784f7h4f8jf09jf30} we obtain:

\begin{THM}
Let $(Y, T_{Y,-}^*)\subset (X, T_{X,-}^*)$ be an embedding of models. Then for each $k\geq 2$, the following diagram commutes:
\begin{align}
\xymatrix{
\ar[dr]\ar[d]\mbox{\emph{\v H}}^1\big(X, \Gc^{(k)}_{T^*_{Y, -}; T^*_{X, -}}\big)
\ar[r] & \mbox{\emph{\v H}}^1\big(X, \Gc^{(k)}_{T^*_{X, -}}\big)\ar[dr]
\\
\ar[dr]\mbox{\emph{\v H}}^1\big(Y, \Gc^{(k)}_{T^*_{Y, -}}\big) &\ar[d] H^1\big(X, \Qcl_{T^*_{Y, -};T^*_{X, -}}^{(k)}\big)\ar[r] & H^1\big(X, \Qcl_{T^*_{X, -}}^{(k)}\big)\ar[d]^{r_*}
\\
&H^1\big(Y, \Qcl_{T^*_{Y, -}}^{(k)}\big)\ar[r]_{i_*} & H^1\big(Y, \Qcl_{T^*_{X, -}}^{(k)}|_Y\big).
}
\label{eduhf89hf983hf983}
\end{align}
\qed
\end{THM}

\noindent
Commutativity of the square in \eqref{eduhf89hf983hf983} and Proposition \ref{rfb4fhf98f80j09fj930j} give:

\begin{THM}\label{tg6fg487fh4hf89j3f03}
Let $\big(\Zc: \Yfr\subset \Xfr\big)$ be a $(k-1)$-split embedding of supermanifolds modelled on  $(Y, T_{Y,-}^*)$ and $(X, T_{X, -}^*)$. Then 
\[
i_*\big(\om(\Yfr)\big) = r_*\big(\om(\Xfr)\big).
\]
where $i_*$ and $r_*$ are the maps in \eqref{eduhf89hf983hf983}.\qed
\end{THM}

\subsection{Embeddings of Splitting Type $(k; k, k+1)$}
Embeddings in a split supermanifold are a particular class of embeddings of splitting type $(k; k, k+1)$.  We single such embeddings out here as their obstruction classes admit a nice characterisation. Consider the diagram on cohomology induced from \eqref{rrtgiug484848hirgjrioiuii}. The piece of relevance for our present purposes is:
\begin{align}
\xymatrix{
\ar[d]^{r_*}H^0\big(X,  \mathcal R^{(k)}_{T^*_{Y, -}; T^*_{X, -}}\big)
\ar[r]^{\dt_1} & H^1\big(X, \Qcl^{(k)}_{T^*_{Y, -}; T^*_{X, -}}\big)\ar[d] \ar[r] & H^1\big(X, \Qcl^{(k)}_{T^*_{X, -}}\big)\ar[d]
\\
H^0\big(Y,  \mathcal R^{(k)}_{T^*_{Y, -}; T^*_{X, -}}|_Y\big)\ar[r]^{\dt_2} &
H^1\big(Y, \Qcl^{(k)}_{T^*_{Y, -}}\big)\ar[r] 
&
H^1(Y, \Qcl^{(k)}_{T^*_{X, -}}|_Y\big)
}
\label{cfgjgjgjgjfjjfhfgbgjf}
\end{align}
We are thus led to the following:

\begin{THM}\label{h89h890309i390}
Let $\big(\Zc: \Yfr\subset \Xfr\big)$ be an embedding of splitting type $(k;k, k+1)$. Then there exists a global section $\vp\in H^0\big(X,  \mathcal R^{(k)}_{T^*_{Y, -}; T^*_{X, -}}\big)$ such that 
\begin{align*}
\dt_1(\vp) = \om(\Zc)&&\mbox{and}&&\dt_2\big(r_*(\vp)\big) = \om(\Yfr).
\end{align*}
\end{THM}

\begin{proof}
Recall, if $\big(\Zc: \Yfr\subset \Xfr\big)$ is an embedding of splitting type $(k;k, k+1)$, then $\Yfr$ and $\Zc$ will be $(k-1)$-split while $\Xfr$ will be $k$-split. In particular, its primary obstruction \emph{as a $(k-1)$-split supermanifold} will vanish (c.f., Theorem \ref{rhf78rhf74hf98fj3f039}). Hence $\om(\Zc)$ will map to zero in $H^1\big(X, \Qcl_{T^*_{X, -}}^{(k)}\big)$ by Proposition \ref{rfb4fhf98f80j09fj930j}; and $\om(\Yfr)$ will map to zero in $H^1\big(X, \Qcl_{T^*_{X, -}}^{(k)}|_Y\big)$ by Theorem \ref{tg6fg487fh4hf89j3f03}. The present theorem now follows from exactness of the rows in \eqref{cfgjgjgjgjfjjfhfgbgjf}. 
\end{proof}

\section{Conormal Sheaves}

\noindent
The rows in the diagram \eqref{rrtgiug484848hirgjrioiuii} will be referred to as \emph{obstruction sequences} associated to an embedding of models $(Y, T^*_{Y, -})\subset (X, T^*_{X, -})$. We refer to the top row as the \emph{ambient obstruction sequence} while the bottom row will be referred to as the \emph{embedded obstruction sequence}. Our objective in this section is to relate these sequences with appropriately twisted, conormal sheaves.

\subsection{Obstruction Sheaves}
We recall here an explicit description of the obstruction sheaves obtained by Green in \cite{GREEN}. To any model $(Z, T^*_{Z, -})$, the obstruction sheaves are given by:
\begin{align}
\Qcl^{(k)}_{T^*_{Z, -}}
\cong 
\left\{
\begin{array}{ll}
\mathcal Hom_{\Oc_Z}\big(T^*_Z,\wedge^kT^*_{Z, -}\big)
&\mbox{if $k$ is even;}
\\
\mathcal Hom_{\Oc_Z}\big(T^*_{Z, -},\wedge^kT^*_{Z, -}\big)
&\mbox{if $k$ is odd.}
\end{array}
\right.
\label{rgf784hf7hf9f40fk04kf48j03j903}
\end{align}
For convenience we use the following notation
\begin{align*}
T^*_{Z, (\pm)^k}
=
\left\{
\begin{array}{ll}
T^*_Z& \mbox{$k$ is even};
\\
T^*_{Z, -} &\mbox{$k$ is odd};
\end{array}
\right.
\end{align*}
Then \eqref{rgf784hf7hf9f40fk04kf48j03j903} can be conveniently stated: 
\begin{align}
\Qcl_{T^*_{X, -}}^{(k)} \cong \mathcal Hom_{\Oc_Z}\big(T^*_{Z, (\pm)^k}, \wedge^k T^*_{Z, -}\big). 
\label{mvnfngbjckdkf}
\end{align}
The obstruction sheaf associated to an embedding of models is however a little more subtle.

\subsection{The Obstruction Sheaf for Embeddings}
Let $f: (Y, T^*_{Y, -})\subset (X, T^*_{X, -})$ be an embedding of models. Recall $f = (i, f^\sharp)$ where $i : Y\subset X$ and $f^\sharp : T^*_{X, -} \ra i_*T^*_{Y, -}\ra0$ (equivalently, $i^*T^*_{X, -} \ra T^*_{Y, -}$ is a surjection). The $k$-th obstruction sheaf associated to $f$ is $\Qcl_{T^*_{Y, -}; T^*_{X, -}}^{(k)}$. In Lemma \ref{rfuefhfiuefeiojoeidep} we found that $\Qcl_{T^*_{Y, -}; T^*_{X, -}}^{(k)}$ is a subsheaf of $\Qcl_{T^*_{X, -}}^{(k)}$. By construction, it pulls back to $\Qcl_{T^*_{Y, -}}^{(k)}$. Hence, we can view $\Qcl_{T^*_{Y, -}; T^*_{X, -}}^{(k)}$ as those sections of $\Qcl_{T^*_{X, -}}^{(k)}$ which pullback to $\Qcl_{T^*_{Y, -}}^{(k)}$. Phrased in this way, $\Qcl_{T^*_{Y, -}; T^*_{X, -}}^{(k)}$ can be seen to satisfy a lifting property. To state it, firstly observe that there exists a natural injection $\Qcl_{T^*_{Y, -}}^{(k)}\ra i^* \Qcl_{T^*_{X, -}}^{(k)}$. This can be deduced from Green's characterisation of the obstruction sheaves in \eqref{mvnfngbjckdkf} combined with the surjection $i^*T^*_{X, (\pm)^k} \ra T^*_{Y, (\pm)^k}$. With this observation we present:

\begin{LIFT}
Let $\Fc$ be a sheaf of $\Oc_X$-modules and suppose $\phi: \Fc \ra \Qcl_{T^*_{X, -}}^{(k)}$ is a morphism such that $i^*\phi$ factors through $\Qcl_{T^*_{Y, -}}^{(k)}\ra i^* \Qcl_{T^*_{X, -}}^{(k)}$, i.e., that there exists a morphism $v : i^*\Fc\ra  \Qcl_{T^*_{Y, -}}^{(k)}$ commuting the following diagram:
\[
\xymatrix{
\ar[d]\Fc\ar[rr]^\phi & & \Qcl_{T^*_{X, -}}^{(k)}\ar[d]
\\
i^*\Fc \ar@{-->}[r]^v \ar@/_1.5pc/@{.>}[rr]_{i^*\phi} & \Qcl_{T^*_{Y, -}}^{(k)} \ar[r] & i^* \Qcl_{T^*_{X, -}}^{(k)}
}
\]
Then $\phi$ factors through the subsheaf $\Qcl^{(k)}_{T^*_{Y, -}; T^*_{X, -}}$. That is, there exists a unique morphism $u: \Fc \ra \Qcl^{(k)}_{T^*_{Y, -}; T^*_{X, -}}$ lifting $v$. In terms of diagrams, the lifting property can be summarised by: given $v$, there exists $u$ commuting the following, 
\begin{align}
\xymatrix{
\ar[d]\Fc \ar@{-->}[dr]^u\ar@/^1pc/@{>}[drr]^\phi
\\
\ar@/_3.5pc/@{>}[drr]_{i^*\phi} i^*\Fc\ar@{-->}[dr]^v& \ar[d] \Qcl^{(k)}_{T^*_{Y, -}; T^*_{X, -}} \ar[r] & \Qcl_{T^*_{X, -}}^{(k)}\ar[d]
\\
& \Qcl^{(k)}_{T^*_{Y, -}}\ar[r] & i^*\Qcl_{T^*_{X, -}}^{(k)}
}
\label{rjf4hf89h49f88fj30j0}
\end{align}
\end{LIFT}
~\\\\
\noindent 
We conclude with the following useful result.

\begin{LEM}\label{rojpjfiorjojfpjepofjep}
Let $\Fc$ be a subsheaf of $\Qcl_{T^*_{X, -}}^{(k)}$ which pulls back to $\Qcl^{(k)}_{T^*_{Y, -}}$. Then $\Fc$ is isomorphic to $\Qcl^{(k)}_{T^*_{Y, -}; T^*_{X, -}}$.  \qed
\end{LEM}

\subsection{The Embedded Obstruction Sequence}
Let $f = (i, f^\sharp):(Y, T_{Y, -}^*)\subset (X, T^*_{X, -})$ be an embedding of models. We set,
\begin{align}
\nu_{Y/X, (\pm)^k}^*
:=
\left\{
\begin{array}{ll}
\Ic_Y/\Ic_Y^2&\mbox{$k$ is even};
\\
i^*K_{T^*_{Y, -}; T^*_{X, -}}&\mbox{$k$ is odd}
\end{array}
\right.
\label{rhf784hf984hf8j03j0}
\end{align}
where $\Ic_Y$ is the ideal sheaf of $Y\subset X$ and $K_{T^*_{Y, -}; T^*_{X, -}}$ is the kernel of the surjection $T_{X, -}^* \ra f^*T^*_{Y, -}$. Since $f$ is holomorphic we have the `conormal bundle sequence',
\[
0 \lra \nu_{Y/X, (\pm)^k}^*\lra i^*T^*_{X, (\pm)^k} \lra T^*_{Y, (\pm)^k}\lra0.
\] 
Since $\wedge^kT^*_{Y, -}$ is locally free, the contravariant functor $\mathcal Hom_{\Oc_Y}\big(-, \wedge^k T^*_{Y, -}\big)$ is exact. We therefore get:
\begin{align}
0
\lra
\mathcal Hom_{\Oc_Y}\big(T^*_{Y, (\pm)^k}, \wedge^k T^*_{Y, -}\big)
\lra
\mathcal H&om_{\Oc_Y}\big(i^*T^*_{X, (\pm)^k},  \wedge^k T^*_{Y, -}\big)
\label{nvnbnfjfdkvnvnfjfj}
\\\notag
&\lra
\mathcal Hom_{\Oc_Y}\big( \nu_{Y/X, (\pm)^k}^*,  \wedge^k T^*_{Y, -}\big)
\lra
0
\end{align}
Note that the left-most term in \eqref{nvnbnfjfdkvnvnfjfj} is isomorphic to $\Qcl_{T^*_{Y, -}}^{(k)}$ by \eqref{mvnfngbjckdkf}. As for the next term observe that, again by \eqref{mvnfngbjckdkf},
\begin{align}
i^*\Qcl_{T^*_{X, -}}^{(k)}
&\cong
i^*\mathcal Hom_{\Oc_X}\big(T^*_{X, (\pm)^k}, \wedge^k T^*_{X, -}\big)
\notag
\\
&\cong 
\mathcal Hom_{\Oc_Y}\big(i^*T^*_{X, (\pm)^k}, i^*\wedge^k T^*_{X, -}\big)
\notag
\\
&\lra
\mathcal Hom_{\Oc_Y}\big(f^*T^*_{X, (\pm)^k}, \wedge^k T^*_{Y, -}\big).
\label{ojppnoioeo}
\end{align}
To explain the map in \eqref{ojppnoioeo} recall that we have the surjection $f^\sharp: i^*T^*_{X, -}\ra T^*_{Y, -}\ra0$. This induces a surjection on exterior powers since the operation of taking exterior powers is right exact. Hence we have a natural transformation of functors $\mathcal Hom_{\Oc_Y}\big(-, i^*\wedge^k T^*_{X, -}\big)\ra \mathcal Hom_{\Oc_Y}\big(-, \wedge^k T^*_{Y, -}\big)$ giving \eqref{ojppnoioeo}. Evidently, we obtain a commutative diagram:
\[
\xymatrix{
\ar[d]_\cong\Qcl_{T^*_{Y, -}}^{(k)} \ar[r] & i^*\Qcl_{T^*_{X, -}}^{(k)}\ar[d]^{\eqref{ojppnoioeo}}
\\
\mathcal Hom_{\Oc_Y}\big(T^*_{Y, (\pm)^k}, \wedge^k T^*_{Y, -}\big) \ar[r] & \mathcal Hom_{\Oc_Y}\big(i^*T^*_{X, (\pm)^k}, \wedge^k T^*_{Y, -}\big).
}
\]
Upon identifying $\Qcl_{T^*_{Y, -}}^{(k)}$ with $i^*\Qcl_{T^*_{Y, -};T^*_{X, -}}^{(k)}$ as sheaves of $\Oc_Y$-modules we conclude:

\begin{PROP}\label{finoppioibuibriv}
Let $f: (Y, T^*_{Y, -})\subset (X, T^*_{X, -})$ be an embedding of models. For each $k$, the  natural transformation $\mathcal Hom_{\Oc_Y}\big(-, i^*\wedge^kT^*_{X, -}\big)\ra \mathcal Hom_{\Oc_Y}\big(-, \wedge^kT^*_{Y, -}\big)$ induces the following morphisim of short exact sequences:
\begin{align*}
\xymatrix{
 \Qcl_{T^*_{Y, -}}^{(k)}\ar[r] \ar[d]_\cong& i^*\Qcl_{T^*_{X, -}}^{(k)}\ar[r]\ar[d] & i^* \mathcal R^{(k)}_{T^*_{Y, -}; T^*_{X, -}}\ar[d]
\\ 
\mathcal Hom_{\Oc_Y}\big(T^*_{Y, (\pm)^k}, \wedge^k T^*_{Y, -}\big)
\ar[r] & 
\mathcal Hom_{\Oc_Y}\big(i^*T^*_{X, (\pm)^k}, \wedge^k T^*_{Y, -}\big)
\ar[r] & 
\mathcal Hom_{\Oc_Y}\big(\nu^*_{Y/X, (\pm)^k}, \wedge^k T^*_{Y, -}\big)
}
\end{align*}
where the isomorphism $ \Qcl_{T^*_{Y, -}}^{(k)}\cong \mathcal Hom_{\Oc_Y}\big(T^*_{Y, (\pm)^k}, \wedge^k T^*_{Y, -}\big)$ comes from \eqref{mvnfngbjckdkf}. \qed
\end{PROP}

\subsection{The Ambient Obstruction Sequence}
In Proposition \ref{finoppioibuibriv} we characterized the embedded obstruction sequence as a sequence of sheaves of $\Oc_Y$-modules. We consider here the ambient obstruction sequence which is a sequence of sheaves $\Oc_X$-modules. Our starting point is the normal bundle sequence of $f:(Y, T^*_{Y, -})\subset (X, T^*_{X, -})$ now as sheaves on $X$:
\[
0 
\lra
N_{Y/X,(\pm)^k}^*
\lra
T^*_{X, (\pm)^k}
\lra 
i_*T^*_{Y, (\pm)^k}
\lra
0
\]
where 
\begin{align}
N_{Y/X,(\pm)^k}^*
=
\left\{
\begin{array}{ll}
\Ic_Y&\mbox{$k$ is even};
\\
K_{T^*_{Y, -}; T^*_{X, -}}&\mbox{$k$ is odd}.
\end{array}
\right.
\label{ojrjiofjroifjoifjoe}
\end{align}
Applying $\mathcal Hom_{\Oc_X}\big(-, \wedge^kT^*_{X, -} \big)$ gives
\begin{align}
0
\lra 
\mathcal Hom_{\Oc_X}\big(i_*T^*_{Y, (\pm)^k}, \wedge^kT^*_{X, -} \big)
\lra
\mathcal H&om_{\Oc_X}\big(T^*_{X, (\pm)^k}, \wedge^kT^*_{X, -} \big)
\label{lmrknkjbbfuyiuhfei}
\\
&\lra
\mathcal Hom_{\Oc_X}\big(N_{Y/X,(\pm)^k}^*, \wedge^kT^*_{X, -} \big)
\lra
0
\notag
\end{align}
The relation to the ambient obstruction sequence is as follows.

\begin{PROP}\label{roiniofioejfiojeiojop}
There exists a morphism of exact sequences,
\begin{align*}
\xymatrix{
\mathcal Hom_{\Oc_X}\big(i_*T^*_{Y, (\pm)^k}, \wedge^kT^*_{X, -}\big)\ar[r]\ar[d]_u
& 
\mathcal Hom_{\Oc_X}\big(T^*_{X, (\pm)^k}, \wedge^kT^*_{X, -}\big)\ar[r]\ar[d]^\cong
&
\mathcal Hom_{\Oc_X}\big(N_{Y/X, (\pm)^k}^*, \wedge^kT^*_{X, -}\big) \ar[d]
\\
\Qcl_{T^*_{Y, -}; T^*_{X, -}}^{(k)} \ar[r] &  \Qcl_{T^*_{X, -}}^{(k)} \ar[r] &  \mathcal R^{(k)}_{T^*_{Y, -}; T^*_{X, -}}
}
\end{align*}
where the above isomorphism comes from \eqref{mvnfngbjckdkf}. 
\end{PROP}

\begin{proof}
The isomorphism $\mathcal Hom_{\Oc_X}\big(T^*_{X, (\pm)^k}, \wedge^kT^*_{X, -}\big)\stackrel{\sim}{\ra} \Qcl_{T^*_{X, -}}^{(k)}$ gives the following composition
\begin{align}
\q:
\mathcal Hom_{\Oc_X}\big(f_*T^*_{Y, (\pm)^k}, \wedge^kT^*_{X, -}\big)
\ra
\mathcal Hom_{\Oc_X}\big(T^*_{X, (\pm)^k}, &\wedge^kT^*_{X, -}\big)
\stackrel{\sim}{\ra} 
\Qcl_{T^*_{X, -}}^{(k)}.
\label{lfkjgnfndjdklkcjvkfk}
\end{align}
 As the embedding $f$ is holomorphic, the pullback $i^*$ defines an exact functor. Now the map $\q$ is injective and so $i^*\q$ is injective giving,
\begin{align}
0\lra  i^*\mathcal Hom_{\Oc_X}\big(i_*T^*_{Y, (\pm)^k}, \wedge^kT^*_{X, -}\big)
\stackrel{i^*\q^\p}{\lra } i^*\Qcl^{(k)}_{T^*_{X, -}}. 
\label{rhf894h9f8hf09j3f90j3}
\end{align}
Now again by holomorphy of $f$ there exists a natural isomorphism $i^*i_*\cong 1$. Using this and the transformation $\mathcal Hom_{\Oc_Y}\big(-, i^*\wedge^kT^*_{X, -}\big)\ra \mathcal Hom_{\Oc_Y}\big(-, \wedge^kT^*_{Y, -}\big)$ yields,
\begin{align}
i^*\mathcal Hom_{\Oc_X}\big(i_*T^*_{Y, (\pm)^k}, \wedge^kT^*_{X, -}\big) 
&\stackrel{\cong}{\lra}
\mathcal Hom_{\Oc_X}\big(i^*i_*T^*_{Y, (\pm)^k}, i^*\wedge^kT^*_{X, -}\big) 
\label{kfmlkrnvbiubfiuuen}
\\
\notag
&\stackrel{\cong}{\lra}
\mathcal Hom_{\Oc_X}\big(T^*_{Y, (\pm)^k}, i^*\wedge^kT^*_{X, -}\big)
\\
\notag
&\lra
\mathcal Hom_{\Oc_X}\big(T^*_{Y, (\pm)^k}, \wedge^kT^*_{Y, -}\big)
\\
\notag
&\stackrel{\cong}{\lra}
\Qcl_{T^*_{Y, -}}^{(k)}.
\end{align}
Injectivity of $i^*\q$ in \eqref{rhf894h9f8hf09j3f90j3} guarantees a morphism $h: i^*\Qcl^{(k)}_{T^*_{X, -}}\ra f^*\Qcl^{(k)}_{T^*_{X, -}}$ commuting the following diagram,\footnote{That there will exist such a commutative diagram can be seen by considering a more abstract setting. Let $A$ and $B$ be algebras with $A\subset B$. Let $A^\p$ be another algebra and suppose we have morphisms $A\stackrel{g}{\ra} A^\p\ra B$. With $g$ we can define a morphism $h : B\ra B$ commuting with $g$ by setting:
\[
h(b) = \left\{
\begin{array}{ll}
g(b) & \mbox{$b \equiv 0\mod A$}
\\
b & \mbox{otherwise}.
\end{array}
\right.
\]
That $h$ is well-defined homomorphism depends essentially on $A$ being a subalgebra of $B$. This is because the condition $b \equiv 0\mod A$ ensures the existence of a unique $a\in A$ mapping to $b$ and so we can identify $b$ with $a$. 
}
\[
\xymatrix{
\ar[d]i^*\mathcal Hom_{\Oc_X}\big(i_*T^*_{Y, (\pm)^k}, \wedge^kT^*_{X, -}\big)\ar[rrr]^{i^*\q}
& & &
i^*\Qcl^{(k)}_{T^*_{X, -}}\ar@{-->}[d]^h
\\
\Qcl_{T^*_{Y, -}}^{(k)} \ar[rrr] & & &i^*\Qcl^{(k)}_{T^*_{X, -}}
}
\]
Hence the morphism $hi^*\q$ factors through $\Qcl_{T^*_{Y, -}}^{(k)}\ra i^*\Qcl^{(k)}_{T^*_{X, -}}$. By \eqref{kfmlkrnvbiubfiuuen}, note that we can write $hi^*\q = i^*\q^{\p}$ for some morphism $\q^{\p} :  \mathcal Hom_{\Oc_X}\big(i_*T^*_{Y, (\pm)^k}, \wedge^kT^*_{X, -}\big)\ra \Qcl^{(k)}_{T^*_{X, -}}$. Then as we have just seen $i^*\q^{\p}$ factors through $\Qcl_{T^*_{Y, -}}^{(k)}\ra i^*\Qcl^{(k)}_{T^*_{X, -}}$. Therefore, by the lifting property (see \eqref{rjf4hf89h49f88fj30j0}), there will exist a morphism $u$, well defined up to isomorphism, commuting the following,
\begin{align}
\xymatrix{
\mathcal Hom_{\Oc_X}\big(i_*T^*_{Y, (\pm)^k}, \wedge^kT^*_{X, -}\big)\ar@{-->}[d]_u
\ar[drrr]^{\q^{\p}} 
\\
\Qcl_{T^*_{Y, -}; T^*_{X, -}}^{(k)}\ar[rrr] && & \Qcl_{T^*_{X, -}}^{(k)}
}
\label{rhf794gf94f8jf09j30}
\end{align}
To obtain the desired morphism of exact sequences we will need to appeal to the universal property of cokernels. In identifying $\mathcal Hom_{\Oc_X}\big(N_{Y/X, (\pm)^k}^*, \wedge^kT^*_{X, -}\big)$ with the cokernel of $\mathcal Hom_{\Oc_X}\big(T^*_{X, (\pm)^k}, \wedge^kT^*_{X, -}\big)\ra\mathcal Hom_{\Oc_X}\big(N_{Y/X, (\pm)^k}^*, \wedge^kT^*_{X, -}\big)$, the universal property guarantees a morphism $\mathcal Hom_{\Oc_X}\big(N_{Y/X, (\pm)^k}^*, \wedge^kT^*_{X, -}\big)\ra \mathrm{coker}~\q^\p$. Combining this with \eqref{rhf794gf94f8jf09j30} we find the following diagram of morphisms:
\begin{align*}
\xymatrix{
\mathcal Hom_{\Oc_X}\big(i_*T^*_{Y, (\pm)^k}, \wedge^kT^*_{X, -}\big)\ar[r]\ar@{=}[d]
& 
\mathcal Hom_{\Oc_X}\big(T^*_{X, (\pm)^k}, \wedge^kT^*_{X, -}\big)\ar[r]\ar[d]^\cong
&
\mathcal Hom_{\Oc_X}\big(N_{Y/X, (\pm)^k}^*, \wedge^kT^*_{X, -}\big) \ar[d]
\\
\ar[d]^u \mathcal Hom_{\Oc_X}\big(i_*T^*_{Y, (\pm)^k}, \wedge^kT^*_{X, -}\big)\ar[r]^{\q^\p} &  \Qcl_{T^*_{X, -}}^{(k)} \ar@{=}[d] \ar[r] &  \mathrm{coker}~\q^\p\ar[d]
\\
\Qcl_{T^*_{Y, -}; T^*_{X, -}}^{(k)}\ar[r] & \Qcl^{(k)}_{T^*_{X, -}}\ar[r] & \mathcal R^{(k)}_{T^*_{Y, -}; T^*_{X, -}}
}
\end{align*}
The proposition now follows.
\end{proof}

\noindent
In putting Proposition \ref{finoppioibuibriv} and \ref{roiniofioejfiojeiojop} together, we have:

\begin{THM}\label{rhf7hf98fh894f03j903j309}
To an embedding of models $f: (Y, T^*_{Y, -})\subset (X, T^*_{X, -})$ we have the following commutative diagram for each $k$,
\begin{align*}
\xymatrix{
\mathcal Hom_{\Oc_X}\big(i_*T^*_{Y, (\pm)^k}, \wedge^kT^*_{X, -}\big)\ar[r]\ar[d]
& 
\mathcal Hom_{\Oc_X}\big(T^*_{X, (\pm)^k}, \wedge^kT^*_{X, -}\big)\ar[r]\ar[d]_\cong
&
\mathcal Hom_{\Oc_X}\big(N_{Y/X, (\pm)^k}^*, \wedge^kT^*_{X, -}\big) \ar[d]
\\
\Qcl_{T^*_{Y, -}; T^*_{X, -}}^{(k)} \ar[r] \ar[d]&  \Qcl_{T^*_{X, -}}^{(k)} \ar[r] \ar[d]&  \mathcal R^{(k)}_{T^*_{Y, -}; T^*_{X, -}}\ar[d]
\\
\Qcl_{T^*_{Y, -}}^{(k)}\ar[r] \ar[d]_\cong& i^*\Qcl_{T^*_{X, -}}^{(k)}\ar[r]\ar[d] & i^* \mathcal R^{(k)}_{T^*_{Y, -}; T^*_{X, -}}\ar[d]
\\ 
\mathcal Hom_{\Oc_Y}\big(T^*_{Y, (\pm)^k}, \wedge^k T^*_{Y, -}\big)
\ar[r] & 
\mathcal Hom_{\Oc_Y}\big(i^*T^*_{X, (\pm)^k}, \wedge^k T^*_{Y, -}\big)
\ar[r] & 
\mathcal Hom_{\Oc_Y}\big(\nu^*_{Y/X, (\pm)^k}, \wedge^k T^*_{Y, -}\big)
}
\end{align*}
\qed
\end{THM}

\subsection{Even Holomorphic Embeddings} 
The vertical arrows in Proposition \ref{finoppioibuibriv} and \ref{roiniofioejfiojeiojop} need not be injective or surjective in general. It is addressing this point which motivates what we term `even' embeddings.

\begin{DEF}\label{dyrgg4hhf3hh3o3io}
\emph{An embedding of models $f = (i, f^\sharp): (Y, T^*_{Y, -})\subset (X, T^*_{X, -})$ is said to be \emph{even} if the surjection of odd conormal sheaves $f^\sharp: f^*T^*_{X, -} \ra T^*_{Y, -}$ is an isomorphism.}
\end{DEF}

\noindent
The following result concerning the embedded obstruction sequence follows straightforwardly from the definition.

\begin{PROP}\label{rgf784f894hf983hf3j0}
To any even embedding $f:(Y, T^*_{Y, -})\subset (X, T^*_{X, -})$, the vertical morphisms in Proposition $\ref{finoppioibuibriv}$ are isomorphisms. 
\qed
\end{PROP}

\noindent
Regarding the ambient obstruction sequence we have similarly:

\begin{PROP}\label{kdmlkcnekjkjejnevne}
To any even embedding $f:(Y, T^*_{Y, -})\subset (X, T^*_{X, -})$, the vertical morphisms in Proposition $\ref{roiniofioejfiojeiojop}$ are isomorphisms. 
\end{PROP}

\begin{proof}
We follow the proof of Proposition \ref{roiniofioejfiojeiojop} more closely. Observe that with the assumption $i^*T^*_{X, -}\cong T_{Y, -}$ we have
\begin{align}
i^*\mathcal Hom_{\Oc_X}\big(i_*T^*_{Y, (\pm)^k}, \wedge^kT^*_{X, -}\big)
\cong \Qcl_{T^*_{Y, -}}^{(k)}.  
\label{rhf89hf89f83jf0j30}
\end{align}
Hence for $\q$ the injection in \eqref{lfkjgnfndjdklkcjvkfk} we see that $i^*\q$ will factor through the isomorphism $v$ in \eqref{rhf89hf89f83jf0j30}. We are thus reduced to the hypotheses in Lemma \ref{rojpjfiorjojfpjepofjep} and can therefore conclude $\mathcal Hom_{\Oc_X}\big(i_*T^*_{Y, (\pm)^k}, \wedge^kT^*_{X, -}\big)$ and $ \Qcl_{T^*_{Y, -}; T^*_{X, -}}^{(k)}$ are isomorphic. The proposition now follows.
\end{proof}

\section{Ideal Sheaves}

\noindent
Supermanifolds of a prescribed splitting type were defined in Definition \ref{rfh89hf98f0j30jf03jf903}. This subsequently inspired the definition of holomorphic embeddings in Definition \ref{fj89hf98f93jf9jf39} from whence we eventually deduce Theorem \ref{rhf7hf98fh894f03j903j309}. Presently, we will describe embeddings by reference to sheaves of ideals.

\subsection{Embeddings of Split Models}
Let $f =(i, f^\sharp): (Y, T^*_{Y, -})\subset (X, T^*_{X, -})$ be an embedding of models. We have the surjections $i^\sharp: \Oc_X \ra i_*\Oc_Y$ and $f^\sharp: T^*_{X, -} \ra i_*T_{Y, -}^*$, where $\Oc_X$ (resp. $\Oc_Y$) is the structure sheaf of $X$ (resp. $Y$). Let $\Ic_Y$ and $K_{T^*_{Y, -}; T^*_{X, -}}$ denote the respective kernels. Since $\wedge^\bt$ is a right-exact functor, the surjection $f^\sharp$ gives $\wedge^\bt T^*_{X, -} \ra \wedge^\bt i_*T^*_{Y, -}\ra0$. To describe the kernel, recall that each exterior power $\wedge^m T^*_{X, -}$ will be a filtered $\Oc_X$-module of length $m$. Denote by $F^m_{~\bt} = \{F^m_{~n}\}_{0\leq n\leq m+1}$ the filtration given by:
\[
0 = F^m_{~0}\subset F^m_{~1} \subset \cdots \subset F^m_{~m}\subset F^m_{~m+1} = \wedge^mT^*_{X, -}.
\]
Successive quotients satisfy,
\[
F^m_{~p+1} / F^m_{~p} \cong \wedge^{m - p}K_{T^*_{Y, -};T^*_{X, -}}\otimes \wedge^p i_* T^*_{Y, -}. 
\]
Hence $F^m_{~m} = \ker\{\wedge^m T^*_{X, -}\ra \wedge^m T^*_{Y, -}\}$ for each $m$. Accordingly, we set
\begin{align}
\Ic_{T^*_{Y, -};T^*_{X, -}} := \Ic_Y \oplus \big(\bigoplus_{m\geq 1} F^m_{~m}\big).
\label{fiuiurhfiuhf983h9f}
\end{align}
Then  $\ker \{\wedge^\bt T^*_{X, -}\ra \wedge^\bt i_*T^*_{Y, -}\}$ is isomorphic to $\Ic_{T^*_{Y, -}; T^*_{X, -}}$ as $\Oc_X$-modules. We list below important properties entertained by $\Ic_{T^*_{Y, -};T^*_{X, -}}$:
\begin{enumerate}[(i)]
	\item $\Ic_{T^*_{Y, -};T^*_{X, -}}$ is $\Zbb$-graded with graded pieces:
	\[
	\Ic_{T^*_{Y, -}; T^*_{X, -}}^j =
	\left\{\begin{array}{ll}
	\Ic_Y&\mbox{$j =0$};
	\\
	K_{T^*_{Y, -}; T^*_{X, -}}&\mbox{$j = 1$}
	\\
	F^j_{~j}&\mbox{$j>0$};
	\end{array}
	\right.
	\]
	\item the grading on $\Ic_{T^*_{Y, -}; T^*_{X, -}}^j$ is induced from the grading on $\wedge^\bt T^*_{X, -}$ in the following sense: if $\xi^j : \wedge^\bt T^*_{X, -} \ra \wedge^j T^*_{X, -}$ denotes the projection onto the $j$-th graded piece, then 
	\[
	F^j_{~j} = \mathrm{im}\left\{\Ic_{T^*_{Y, -}; T^*_{X, -}} \hookrightarrow \wedge^\bt T^*_{X, -} \stackrel{\xi^j}{\lra} \wedge^j T^*_{X, -}\right\};
	\]
	\item the quotient $\wedge^\bt T^*_{X, -}/ \Ic_{T^*_{Y, -}; T^*_{X, -}}$ is isomorphic to the sheaf of exterior algebras $\wedge^\bt i_*T^*_{Y, -}$. 
\end{enumerate}
We view $\Ic_{T^*_{Y, -};T^*_{X, -}}$ as the ideal sheaf defining the embedding $e_{(Y, T^*_{Y, -})}\subset  e_{(X, T^*_{X, -})}$, of split models.

\subsection{Embeddings in Split Models}
Based on the observations (i), (ii) and (iii) made earlier, we propose the following general definition of embeddings \emph{in} a split supermanifold $e_{(X, T^*_{X, -})}$.

\begin{DEF}\label{rhf9h49f89f83jf039j}
\emph{To an embedding of models $(Y, T^*_{Y, -})\subset (X, T_{X,-}^*)$, let $\Ic\subset \wedge^\bt T^*_{X, -}$ be a sheaf of ideals satisfying:
\begin{enumerate}[(i)]
	\item $\Ic$ is $\Zbb_2$-graded, with grading inherited from $\wedge^\bt T^*_{X, -}$ in the following sense: let $\xi^\pm : \wedge^\bt T^*_{X, -} \ra \wedge^\pm T^*_{X, -}$ be the projection onto the even and odd graded components.\footnote{as $\Oc_X$-modules we have $\wedge^+T^*_{X, -} = \oplus_{j\geq0} \wedge^{2j}T^*_{X, -}$ and $\wedge^-T^*_{X, -} = \oplus_{j\geq0} \wedge^{2j+1}T^*_{X, -}$.} Note that $\wedge^+T^*_{X, -}\subset \wedge^\bt T^*_{X, -}$ is a commutative subalgebra and $\wedge^-T^*_{X, -}$ is an $\wedge^+T^*_{X, -}$-module. Set, 
	\[
	\Ic^\pm = \mathrm{im}\left\{ \Ic \hookrightarrow \wedge^\bt T^*_{X, -} \stackrel{\xi^\pm}{\lra} \wedge^\pm T^*_{X, -}\right\}.
	\]
	Then as $\wedge^+T^*_{X, -}$-modules we have $\Ic \cong \Ic^+\oplus \Ic^-$. We refer to $\Ic^+$ resp. $\Ic^-$ as the even and odd graded components of $\Ic$;
	\item modulo the fermionic ideal $\Jc_{T^*_{X, -}}^2$,
	\begin{align*}
	\Ic^+ \mod\Jc_{T^*_{X, -}}^2 = \Ic_Y
	&&
	\mbox{and}
	&&
	\Ic^-\mod \Jc_{T^*_{X, -}}^2 = K_{T^*_{Y, -}; T^*_{X, -}}.
	\end{align*}
	\item $\wedge^\bt T^*_{X, -}/\Ic$ and $\wedge^\bt i_*T^*_{Y, -}$ are locally isomorphic.
\end{enumerate}
If $\Ic$ satisfies (i), (ii) and (iii) above then it will be called an \emph{ideal sheaf for an embedding of supermanifolds over $(Y, T^*_{Y, -})\subset (X, T_{X,-}^*)$}.}
\end{DEF}

\noindent
Definition \ref{rhf9h49f89f83jf039j} is made precisely to capture the following correspondence:
\begin{align}
\left\{
\begin{array}{l}
\mbox{Ideal sheaves $\Ic$ for holo--} 
\\
\mbox{morphic embeddings over}\\
\mbox{$(Y, T^*_{Y, -})\subset (X, T_{X,-}^*)$}
\end{array}
\right\}
\Longleftrightarrow
\left\{
\begin{array}{l}
\mbox{Holomorphic embeddings}
\\
\mbox{$(\Zc: \Yfr\subset e_{(X, T_{X,-}^*)})$ for $\Yfr$}
\\
\mbox{modelled on $(Y, T_{Y, -}^*)$}
\end{array}
\right\}
\label{008h489fh94f94fj4}
\end{align}
~\\

\subsection{Splitting of Submanifolds}
Let $\big(\Zc: \Yfr\subset e_{(X, T^*_{X, -})}\big)$ be defined by an ideal sheaf $\Ic$. Then (using \eqref{rhf4hf984j09fjfj30} and the notation in \eqref{ojrjiofjroifjoifjoe}) from Definition \ref{rhf9h49f89f83jf039j}(ii) we have morphisms of exact sequences:
\begin{align}
\xymatrix{
\Ic^\pm \ar[r] \ar[d]& \wedge^\pm T^*_{X, -} \ar[d]\ar[r] & i_*\Oc^\pm_\Yfr\ar[d]
\\
N^*_{Y/X, \pm} \ar[r] & \wedge^\pm T^*_{X, -}/ \Jc_{T^*_{X, -}}^2\ar[r] \ar@/^/@{.>}[u] & i_* \big(\Oc^\pm_\Yfr/\Jc_\Yfr^2\big)
}
\label{rhf894hf984jf90jf093}
\end{align}
A straightforward application of Lemma \ref{rbf8gf87h4f98h389fj0j30} gives:

\begin{LEM}\label{rnuf4f4fj4jfjf0390fj930jf0}
Let $\Ic$ be a sheaf of ideals defining an embedding $\big(\Zc: \Yfr\subset e_{(X, T^*_{X, -})}\big)$. Suppose there exist $\Oc_X$-module morphisms $N^*_{Y/X, \pm} \ra \Ic^\pm$ which commute with the natural inclusions $\wedge^\pm T^*_{X, -}/\Jc^2_{T^*_{X, -}}\ra  \wedge^\pm T^*_{X, -}$ represented by the dotted arrow in \eqref{rhf894hf984jf90jf093}. Then,
\begin{enumerate}[(i)]
	\item $\Yfr$ is split;
	\item $\Ic\cong \Ic_{T^*_{Y, -}; T^*_{X, -}}$. 
\end{enumerate}
\end{LEM}

\begin{proof}
Part \emph{(i)} is immediate. As for \emph{(ii)}, consider that the splitting for $\Yfr$ gives an isomorphism $\psi : i_*\Oc_\Yfr \stackrel{\sim}{\ra} \wedge^\bt i_*T^*_{Y, -}$. This splitting is induced from an automorphism $\psi^\p$ of $\wedge^\bt T^*_{X, -}$. Hence we obtain the following morphism of exact sequences,
\[
\xymatrix{
0\ar[r] & \Ic\ar[d]_{\psi^{\p\p}}\ar[r] & \wedge^\bt T^*_{X, -} \ar[r] \ar[d]_{\psi^\p} & \Oc_\Yfr \ar[d]_{ \psi}\ar[r] & 0
\\
0\ar[r] & \Ic_{T^*_{Y, -}; T^*_{X, -}}\ar[r] &\wedge^\bt T^*_{X, -} \ar[r] & \wedge^\bt i_*T^*_{Y, -}\ar[r] & 0
}
\]
Since $\psi$ and $\psi^\p$ are isomorphisms, so is $\psi^{\p\p}$.
\end{proof}

\subsection{The Maximal Splitting Degree}
From Lemma \ref{rnuf4f4fj4jfjf0390fj930jf0} it is clear that the space $\mathrm{Hom}_{\Oc_X}\big(N_{Y/X, \pm}^*, \Ic^\pm\big)$ is an important invariant of the ideal sheaf $\Ic$. It is however a little too large for our purposes. We consider instead a subset defined as follows. Firstly observe from Definition \ref{rhf9h49f89f83jf039j}(ii) that any morphism in $\mathrm{Hom}_{\Oc_X}\big(N_{Y/X, \pm}^*, \Ic^\pm\big)$ will give a morphism $N_{Y/X, \pm}^*\ra N_{Y/X, \pm}^*$ modulo $\Jc_{T^*_{X, -}}^2$. Set,
\[
\widetilde\Hom\big(N_{Y/X, \pm}^*, \Ic^\pm\big)
:=
\left\{
F^\pm \in \Hom\big(N_{Y/X, \pm}^*, \Ic^\pm\big)
\colon 
F^\pm\mod \Jc_{T^*_{X, -}}^2 =\id_{N_{Y/X, \pm}^*}
\right\}.
\]
To each $F^\pm\in \widetilde\Hom_{\Oc_X}\big(N_{Y/X, \pm}^*, \Ic^\pm\big)$ consider the composition
\[
\xi^j(F^\pm) : N_{Y/X, \pm}^* \stackrel{F^\pm}{\lra} \Ic^\pm \hookrightarrow \wedge^\bt T^*_{X, -} \stackrel{\xi^j}{\lra} \wedge^j T^*_{X, -}
\]
where $\xi^j : \wedge^\bt T^*_{X, -}\ra \wedge^j T^*_{X, -}$ is the projection. We define the `maximal splitting degree' of $F^\pm$ as follows.

\begin{DEF}\label{8r8rfgh8hg8hg9004}
\emph{Let $F^\pm\in  \widetilde{\mathrm{Hom}}_{\Oc_X}\big(N_{Y/X, \pm}^*, \Ic^\pm\big)$. The \emph{maximal splitting degree} of $F^\pm$, denoted $m_{F^\pm}$, is defined to be:
\[
m_{F^\pm} 
:=
\max\big\{ m \colon \xi^{m^\p}(F^\pm) = 0 ~\forall ~2\leq m^\p\leq m\big\}
\]
Note that $m_{F^\pm} \geq 2$. If $m_{F^\pm}$ coinsides with $\mathrm{rank}~T^*_{X, -}$, we will set $m_{F^\pm} = \8$.}
\end{DEF}

\noindent
We can formulate statements about splitting now in terms splitting degrees. 

\begin{LEM}\label{krknjviurvburnvien}
Let $\Ic$ be an ideal sheaf defining an embedding $\big(\mathcal Z: \Yfr\subset e_{(X, T^*_{X, -})}\big)$. Then  $\Yfr$ is split if and only if there exist homomorphisms $F^\pm\in \widetilde{\mathrm{Hom}}_{\Oc_X}\big(N_{Y/X, \pm}^*, \Ic^\pm\big)$ with maximal splitting degree $m_{F^\pm} = \8$. 
\end{LEM}

\begin{proof}
In the converse direction, the existence of $F^\pm$ with $m_{F^{\pm}} =\8$ implying $\Yfr$ is split is a restatement of Lemma \ref{rnuf4f4fj4jfjf0390fj930jf0}. In the other direction, suppose now that $\Yfr$ is split. We will then obtain an inclusion $f_*\big(\Oc^\pm_\Yfr/ \Jc_\Yfr^2\big) \ra \Oc_\Yfr^\pm$ commuting with the natural inclusion $\wedge^\pm T^*_{X, -}/ \Jc_{T^*_{X, -}}^2 \ra \wedge^\pm T^*_{X, -}$. This induces homomorphisms $F^\pm$ with maximal splitting degree $m_{F^\pm} = \8$. 
\end{proof}

\subsection{Ideal Sheaves and Embeddings}
In \eqref{008h489fh94f94fj4} we claimed a correspondence between ideal sheaves and holomorphic embeddings. In this section we clarify this claim. Fix a system of generators $\widehat F$ for $\Ic_{T^*_{Y, -}; T^*_{X, -}}$ with $m_{\widehat F} = \8$. Now consider the set 
\[
\mathscr S^{\geq k}_{T^*_{Y, -}; T^*_{X, -}}(\widehat F)
:=
\big\{
\big(\Ic, F\big)
\colon m_F\geq k\big\}
\]
for $\Ic$ an ideal sheaf for an embedding over $(Y, T^*_{Y, -})\subset (X, T^*_{X, -})$ and $F$ a system of generators for $\Ic$. We consider $\mathscr S^{\geq k}_{T^*_{Y, -}; T^*_{X, -}}(\widehat F)$ as a pointed set with base-point $\big(\Ic_{T^*_{Y, -}; T^*_{X, -}}, \widehat F\big)$. Clearly $\mathscr S^{\geq k+1}_{T^*_{Y, -}; T^*_{X, -}}(\widehat F)\subset \mathscr S^{\geq k}_{T^*_{Y, -}; T^*_{X, -}}(\widehat F)$ is an inclusion of pointed sets. We have:

\begin{THM}\label{lfvrnvnrnvrnvknvlke}
Let $(Y, T^*_{Y, -})\subset (X, T^*_{X, -})$ be an embedding of models. Then there exist maps commuting the following diagram:
\begin{align}
\xymatrix{
\mathscr S^{\geq k+1}_{T^*_{Y, -}; T^*_{X, -}}(\widehat F)
\ar[d]
\ar[rr] & &
\mathscr S^{\geq k}_{T^*_{Y, -}; T^*_{X, -}}(\widehat F)
\ar[d]
\\
\mbox{\emph{\v H}}^1\big(X, \Gc^{(k+1)}_{T^*_{Y, -};T^*_{X,-}}\big)
\ar[rr]
& &
\mbox{\emph{\v H}}^1\big(X, \Gc^{(k)}_{T^*_{Y, -}; T^*_{X,-}}\big)
}
\label{cejiejieeeeoejofpekijeoi}
\end{align}
\end{THM}

\begin{proof}
We will construct a map $\mathscr S^{\geq k}_{T^*_{Y, -}; T^*_{X, -}}(\widehat F) \ra \mbox{\v H}^1\big(X, \Gc^{(k)}_{T^*_{X,-}}\big)$ from whence this theorem will follow. We begin with the following observation: to any ideal sheaf $\Ic$ defining an embedding of supermanifolds, note that a choice of generators $F$ for $\Ic$ will give morphisms $F^\pm\in \widetilde{\mathrm{Hom}}_{\Oc_X}\big(N_{Y/X, \pm}^*, \Ic^\pm\big)$. Here $F^+$ resp. $F^-$ are the even and odd components of $F$. Modulo $\Jc_{T^*_{X, -}}^2$, $F$ generates $\Ic_Y$ resp. $K_{T^*_{Y, -}; T^*_{X, -}}$. We set $m_F = \min\{m_{F^+}, m_{F^-}\}$. Thus to generators $F$ of $\Ic$ we have a homomorphism $h_F \in \Hom_{\Oc_X} \big(N_{Y/X, (\pm)^{m_F}} , \wedge^{m_F}T^*_{X, -}\big)$. Note, 
\begin{align}
\big(h_F= 0\big) \iff \big(m_F = \8\big) \iff \big(m_{F^\pm} = \8\big) \iff \big(\Yfr~\mbox{is split}\big),
\label{rhf4gf78hf983jf389fh39hf3}
\end{align}
the latter implication following from  Lemma \ref{krknjviurvburnvien}. Now let $(\Ic, F)\in \mathscr S^{\geq k}_{T^*_{Y, -}; T^*_{X, -}}(\widehat F)$. Then associated to $(\Ic, F)$ is the morphism $h_F\in \Hom_{\Oc_X} \big(N_{Y/X, (\pm)^{m_F}} , \wedge^{m_F}T^*_{X, -}\big)$. Recall that there exists a surjective morphism of sheaves
\[
\Qcl_{T^*_{X, -}}^{(m_F)} \lra \mathcal Hom_{\Oc_X} \big(N_{Y/X, (\pm)^{m_F}} , \wedge^{m_F}T^*_{X, -}\big)
\lra0.
\] 
Hence over each open set $U\subset X$ there will exist some $\nu_U\in \Qcl_{T^*_{X, -}}^{(m_F)}(U)$ mapping to $h_F|_U$. At this stage we recall the following short exact sequence relating (even) derivations and obstruction sheaves from \cite{ONISHCLASS}:
\[
0 
\lra 
\Qcl_{T^*_{X, -}}^{(2k+1)}
\lra 
\mathfrak g^{(2k)}_{T^*_{X, -}}
\lra
\Qcl_{T^*_{X, -}}^{(2k)}
\lra 
0
\]
where $\mathfrak g^{(2k)}_{T^*_{X, -}}$ is the sheaf of derivations of $\wedge^\bt T^*_{X, -}$ sending $\wedge^\ell T^*_{X, -} \ra \wedge^{\ell + 2k} T^*_{X, -}$. We can therefore deduce that over each open set $U$, there will exist a derivation $\dt_U\in \mathfrak g^{(2k)}_{T^*_{X, -}}(U)$ mapping to $h_F|_U$. Now, the sheaf of derivations $\mathfrak g_{T^*_{X, -}}$ is a nilpotent Lie algebra with $\mathfrak g^{(k)}_{T^*_{X, -}}$ its $k$-th graded component. Exponentiating defines a bijection $\mathfrak g^{(k)}_{T^*_{X, -}}\cong \mathcal G^{(k)}_{T^*_{X, -}}$ as sheaves of sets. We can therefore exponentiate the derivation $\dt_U$ to an automorphism $e^{\dt_U}\in \Gc^{(m_F)}_{T^*_{X, -}}(U)$. Now let $F^\p_U = e^{-\dt_U}F|_U$. Then $F^\p_U$ will generate a sheaf of ideals $\Ic^\p(U)$ in $\wedge^\bt T^*_{X, -}(U)$. Since we know $\dt_U \mapsto h_F|_U$ it follows that,
\[
m_{F^\p_U} \geq  m_F + 1. 
\]
Hence over an open set $U\subset X$ we can find an automorphism lifting the maximal splitting degree $m_F$ of $F$. Iterating this procedure, it is clear that we can find local automorphisms $\al_U$ such that $m_{\al_U (F|_U)} = \8$. Uniqueness of $\al_U$ is established by requiring $\al_U(F|_U) = \widehat F|_U$. Observe that $\al_U$ will induce an automorphism of $\wedge^\bt T^*_{Y, -}(U\cap Y)$, so it therefore lies in $\Gc^{(m_F)}_{T^*_{Y, -}; T^*_{X, -}}(U)$. In this way we can thus assign to any $(\Ic, F)\in \mathscr S^{\geq k}_{T^*_{Y, -}; T^*_{X, -}}(\widehat F)$ and covering $\Ufr$ of $X$, a $0$-cochain $\underline\al = \big\{\al_U\big\}_{U\in \Ufr} \in C^0\big(\Ufr, \Gc^{(k)}_{T^*_{Y, -}; T^*_{X, -}}\big)$. The map sending $(\Ic, F)$ to the class of $\big\{\al_U\al_V^{-1}\big\}_{U, V\in \Ufr}$ defines a map of pointed sets $\mathscr S^{\geq k}_{T^*_{Y, -}; T^*_{X, -}}(\widehat F) \ra \mbox{\v H}^1\big(X, \Gc^{(k)}_{T^*_{X,-}}\big)$. 
\end{proof}

\subsection{Relation to Obstruction Classes}
In the diagram of sheaves in Theorem \ref{rhf7hf98fh894f03j903j309}, the rows are short exact sequences. Hence they give long exact sequences on cohomology.  Observe then that we have a map,
\[
\Theta: \mathrm{Hom}_{\Oc_X}\big( N_{Y/X, (\pm)^k}, \wedge^k T^*_{X, -}\big) 
\lra 
H^0\big(X, \mathcal R^{(k)}_{T^*_{Y, -}; T^*_{X, -}}\big)
\lra
H^1\big(X, \Qcl_{T^*_{Y, -}; T^*_{X, -}}^{(k)}\big).
\]
Thus to any $(\Ic, F) \in \mathscr S^{\geq k}_{T^*_{Y, -}; T^*_{X, -}}(\widehat F)$ we have $\Theta(\Ic, F) := \Theta(h_F) \in H^1\big(X, \Qcl_{T^*_{Y, -}; T^*_{X, -}}^{(k)}\big)$. Hence a map $\mathscr S^{\geq k}_{T^*_{Y, -}; T^*_{X, -}}(\widehat F)\ra H^1\big(X, \Qcl_{T^*_{Y, -}; T^*_{X, -}}^{(k)}\big)$. Now note from Theorem \ref{lfvrnvnrnvrnvknvlke} that we also have the composition $\mathscr S^{\geq k}_{T^*_{Y, -}; T^*_{X, -}}(\widehat F) \ra \mbox{\v H}^1\big(X, \Gc^{(k)}_{T^*_{X,-}}\big) \stackrel{\om_*}{\ra} H^1\big(X, \Qcl_{T^*_{Y, -}; T^*_{X, -}}^{(k)}\big)$ where $\om_*$ is the map sending $(\mathcal Z: \Yfr\subset \Xfr\big) \mapsto \om(\Zc)$. When the embedding of models is even, we can identify $\Theta$ and $\om_*$ since, by Proposition \ref{kdmlkcnekjkjejnevne}, the map $\mathrm{Hom}_{\Oc_X}\big( N_{Y/X, (\pm)^k}, \wedge^k T^*_{X, -}\big) \ra H^0\big(X, \mathcal R^{(k)}_{T^*_{Y, -}; T^*_{X, -}}\big)$ is an isomorphism. Thus:

\begin{THM}\label{rf8hg984h9g8gj4j0}
Let $(Y, T^*_{Y, -})\subset (X, T^*_{X, -})$ be an even embedding of models. Then the following diagram commutes,
\begin{align}
\xymatrix{
\mathscr S^{\geq k}_{T^*_{Y, -}; T^*_{X, -}}
\ar[d]
\ar[rr]^{\Theta}&& 
H^1\big(X, \Qcl_{T*_{Y, -}; T^*_{X, -}}^{(k)}\big)
\ar@{=}[d]
\\
\mbox{\emph{\v H}}^1\big(X, \Gc^{(k)}_{T^*_{Y, -}; T^*_{X,-}}\big)
\ar[rr]^{\om_*}
&&
H^1\big(X, \Qcl_{T^*_{Y, -}; T^*_{X, -}}^{(k)}\big)
}
\label{cejiejioejokk4fmkrfpekijeoi}
\end{align}\qed
\end{THM}

\section{Applications I: Some Generalities}
\label{rfu4fhfoj4oifi4fmp}

\subsection{General Characterisations}
In Theorem \ref{h89h890309i390} we deduced, for embeddings $\big(\Zc: \Yfr\subset \Xfr\big)$ over $(Y, T^*_{Y, -})\subset (X, T^*_{X, -})$ of splitting type $(k; k, k+1)$, the existence of global sections in $H^0\big(X, \mathcal R^{(k)}_{T^*_{Y, -}; T^*_{X, -}}\big)$ which map to the obstruction classes of $\Zc$ and $\Yfr$ respectively. From Theorem \ref{rf8hg984h9g8gj4j0} we can see what these classes are explicitly if $\Xfr = e_{(X,T^*_{X, -})}$ is the split model and the embedding of models $(Y, T^*_{Y, -})\subset (X, T^*_{X, -})$ is even. They can be derived from the generators of ideal sheaves. Now concerning even embeddings more generally, we have recourse to Proposition \ref{kdmlkcnekjkjejnevne} which justifies studying only the top and bottom rows of the diagram in Theorem \ref{rhf7hf98fh894f03j903j309}, which are appropriately twisted sequences of conormal sheaves associated to the embedding $(Y, T^*_{Y, -})\subset (X, T^*_{X, -})$. We can then partially address the splitting question: \emph{let $\Yfr$ be a supermanifold modelled on $(Y, T^*_{Y, -})$. Is it split?}

\begin{THM}\label{ldfjnvnjfjdkkskskdkd}
Let $(Y, T^*_{Y, -})$ be a model and suppose there exists an even embedding $f : (Y, T^*_{Y, -})\subset (X, T^*_{X, -})$ such that, for all $k$, either:
\begin{enumerate}[(i)]
	\item $\Hom_{\Oc_Y}\big( \nu^*_{Y/ X, (\pm)^k}, \wedge^k T^*_{Y, -}\big) = (0)$ or;
	\item $\Hom_{\Oc_X}\big( N^*_{Y/ X, (\pm)^k}, \wedge^k T^*_{X, -}\big) = (0)$.
\end{enumerate}
Then any supermanifold $\Yfr$ modelled on $(Y, T^*_{Y, -})$ which can be embedded in the split model $e_{(X, T^*_{X, -})}$ will be split.
\end{THM}

\begin{proof}
To prove \emph{(i)}, suppose the hypotheses of the theorem and let $\big(\Zc: \Yfr \subset e_{(X, T^*_{X, -})}\big)$ be an embedding of $\Yfr$ in $e_{(X, T^*_{X, -})}$. By Theorem \ref{h89h890309i390} and Proposition \ref{kdmlkcnekjkjejnevne}, any obstruction to splitting $\Yfr$ will come from $\Hom_{\Oc_X}\big( \nu_{Y/ X, (\pm)^k}, \wedge^k T^*_{Y, -}\big)$, which vanishes by assumption. Hence any obstruction to splitting $\Yfr$ vanishes which means $\Yfr$ must be split. We can deduce part \emph{(ii)} similarly by reference to Proposition \ref{rfb4fhf98f80j09fj930j} and Theorem \ref{h89h890309i390}. 
\end{proof}

\subsection{Split Embeddings of Models}
The notion of splitness has been defined for supermanifolds in Definition \ref{rfh89hf98f0j30jf03jf903} and embeddings of supermanifolds in Definition \ref{rufuirbiuehfuheoifeojfoe}. Presently, Theorem \ref{ldfjnvnjfjdkkskskdkd} motivates the following definition of splitness for embeddings of models. 

\begin{DEF}\label{finioruhf4jfio4jo4joji4}
\emph{An embedding of models $f : (Y, T^*_{Y, -}) \subset(X, T^*_{X, -})$ is said to be \emph{split} if every holomorphic submanifold of the split model $e_{(X, T^*_{X, -})}$ over $f$ is split as a supermanifold.}
\end{DEF}

\noindent
In the above terminology, Theorem \ref{ldfjnvnjfjdkkskskdkd} asserts: an even embedding of models $f : (Y, T^*_{Y, -}) \subset(X, T^*_{X, -})$ is split if either Theorem \ref{ldfjnvnjfjdkkskskdkd}\emph{(i)} or Theorem \ref{ldfjnvnjfjdkkskskdkd}\emph{(ii)} hold for each $k\geq2$. These conditions can be relaxed slightly since what is ultimately of importance is the image of the Hom spaces in Theorem \ref{ldfjnvnjfjdkkskskdkd}\emph{(i)} and \emph{(ii)} in the obstruction spaces $H^1\big(X, \Qcl_{T^*_{Y, -}; T^*_{X, -}}^{(k)}\big)$ and $H^1\big(X, \Qcl_{T^*_{Y, -}}^{(k)}\big)$ respectively. This leads to the following.

\begin{PROP}\label{rfh47fg784f78hf98j30fj30}
Let $f : (Y, T^*_{Y, -}) \subset(X, T^*_{X, -})$  be an even embedding of models and suppose the normal bundle sequence of the embedding of spaces $Y\subset X$ is split. Then $f$ is split.
\end{PROP}

\begin{proof}
Since $f$ is even, 
\[
\nu^*_{Y/X, (\pm)^k}
=
\left\{
\begin{array}{ll}
\Ic_Y/\Ic_Y^2
& 
\mbox{$k$ is even}
\\
(0)
&
\mbox{$k$ is odd}.
\end{array}
\right.
\]
Now as $T^*_{Y, -}$ is locally free, so are its exterior powers. In particular, assuming the normal bundle sequence of $Y\subset X$ splits, the sequence of sheaves on the bottom row of Theorem \ref{rhf7hf98fh894f03j903j309} is split exact for $k$ even. It is split exact when $k$ is odd since $\nu^*_{Y/X,-} = (0)$, as stated above. On cohomology, the image of $\Hom_{\Oc_Y}\big( \nu^*_{Y/ X, (\pm)^k}, \wedge^k T^*_{Y, -}\big)$ in $H^1\big(X, \Qcl_{T^*_{Y, -}}^{(k)}\big)$ vanishes for all $k$. By Theorem \ref{h89h890309i390} and Proposition \ref{rgf784f894hf983hf3j0}, the embedding of models $f$ will be split.
\end{proof}

\begin{EX}\label{9008f9gf73f}
A classical result of Van de Ven in \cite{VDV} states that the normal bundle sequence of any linear subspace $i:Y\subset \Pbb_\Cbb^m$ will be split exact. Hence by Proposition $\ref{rfh47fg784f78hf98j30fj30}$, any embedding of models $f = (i, f^\sharp) : (Y, i^*T^*_{\Pbb_\Cbb^m, -})\subset (\Pbb^m_\Cbb, T^*_{\Pbb^m_\Cbb, -})$ will be split, for $T^*_{\Pbb^m_\Cbb, -}$ any locally free sheaf on $\Pbb^m_\Cbb$.
\end{EX}

\noindent
In a subsequent section We will recover the results of Example \ref{9008f9gf73f} in a particular case without reference to Van de Ven's result.

\subsection{Embeddings in Projective Spaces and Twistings}
We continue here from where we left off in Example \ref{9008f9gf73f}. Let $X = \Pbb^m_\Cbb$ and $Y$ be a holomorphic submanifold. Denote by $i : Y\subset \Pbb_\Cbb^m$ the holomorphic embedding of spaces. For any locally free sheaf $T_{\Pbb_\Cbb^m, -}^*$ we get an even embedding of models $f: (Y, i^*T^*_{\Pbb_\Cbb^m, -})\subset (\Pbb^m, T^*_{\Pbb_\Cbb^m, -})$. Note that this embedding can be `twisted'  by replacing $T^*_{\Pbb_\Cbb^m}$ with $T^*_{\Pbb_\Cbb^m, -}(\ell) := T^*_{\Pbb_\Cbb^m}\otimes \Oc_{\Pbb_\Cbb^m}(\ell)$. We denote by $f^\ell : (Y, i^*T^*_{\Pbb_\Cbb^m, -}(\ell))\subset (\Pbb_\Cbb^m, T^*_{\Pbb^m_\Cbb, -}(\ell))$ the embedding of models obtained by twisting $f$. In the case where $Y = \Pbb^{m^\p}_\Cbb$ is also a projective space for some $m^\p\leq m$, a famous theorem of Serre can be applied to deduce the existence of split embeddings of models.  

\begin{THM}\label{roponovbueib}
Let $f: (\Pbb^{m^\p}_\Cbb, i^*T^*_{\Pbb_\Cbb^m, -})\subset (\Pbb_\Cbb^m, T^*_{\Pbb_\Cbb^m, -})$ be an embedding of models, where $i : \Pbb^{m^\p}_\Cbb\subset \Pbb_\Cbb^m$ is holomorphic. Then there exists some integer $\ell_0$ such that $f^\ell$ is split for all $\ell\leq \ell_0$. 
\end{THM}

\begin{proof}
This result relies on Serre's Theorem B and boundedness of the exterior algebra as a graded commutative algebra. Recall that Serre's Theorem B, as stated in \cite{OSS}, implies: for any locally free sheaf on $\Fc$ on projective space $\Pbb^j_\Cbb$, there exists some $\ell_0$ such that $H^0(\Pbb_\Cbb^j, \Fc(\ell)) = (0)$ for all $\ell\leq \ell_0$. 
Now with $f^\ell : (\Pbb^{m^\p}_\Cbb, i^*T^*_{\Pbb_\Cbb^m, -}(\ell))\subset (\Pbb_\Cbb^m, T^*_{\Pbb_\Cbb^m, -}(\ell))$ see that for each $k$ the corresponding conormal bundle of $f^\ell$, denoted  $\nu^*_{\Pbb^{m^\p}_\Cbb/\Pbb^m_\Cbb, (\pm)^k}(\ell)$, is:
\[
\nu^*_{\Pbb^{m^\p}_\Cbb/\Pbb^m_\Cbb, (\pm)^k}(\ell)
=
\left\{
\begin{array}{ll}
\Ic_{\Pbb^{m^\p}_\Cbb}/\Ic_{\Pbb^{m^\p}_\Cbb}^2
&
\mbox{$k$ is even};
\\
(0)
&
\mbox{$k$ is odd}.
\end{array}
\right. 
\]
Since the embedding $i$ is holomorphic, $\nu^*_{\Pbb^{m^\p}_\Cbb/\Pbb^m_\Cbb, (\pm)^k}(\ell)$ will be locally free. Now, for any locally free sheaf $\Fc$ on $\Pbb^j_\Cbb$ we have $\wedge^k \big(\Fc\otimes \Oc_{\Pbb^j_\Cbb}(\ell)\big) = \wedge^k \Fc \otimes \Oc_{\Pbb_\Cbb^j}(k\ell)$. Hence for each $k$, 
\[
(f^\ell)^*\mathcal R^{(k)}_{T^*_{\Pbb^{m^\p}_\Cbb, -}(\ell);T^*_{\Pbb^m_\Cbb, -}(\ell)}
\cong
\left\{
\begin{array}{ll}
\mathcal Hom_{\Oc_{\Pbb^{m^\p}_\Cbb}}\big( \Ic_{\Pbb^{m^\p}_\Cbb}/\Ic_{\Pbb^{m^\p}_\Cbb}^2, \wedge^k T^*_{\Pbb^{m^\p}_\Cbb, -}\big)\otimes \Oc_{\Pbb^{m^\p}_\Cbb}(k\ell)
&
\mbox{$k$ is even;}
\\
(0)
&
\mbox{$k$ is odd}
\end{array}
\right.
\]
where $T^*_{\Pbb_\Cbb^{m^\p}, -} = i^*T^*_{\Pbb^m_\Cbb}$. Since $\nu^*_{\Pbb^{m^\p}_\Cbb/\Pbb^m_\Cbb, (\pm)^k}(\ell)$ is locally free, so is $(f^\ell)^*\mathcal R^{(k)}_{T^*_{\Pbb_\Cbb^{m^\p}, -}(\ell);T^*_{\Pbb^m_\Cbb, -}(\ell)}$ and we can apply Serre's Theorem B. It implies  there exists an $\ell(k)$ such that $H^0\big(\Pbb^{m^\p}_\Cbb, (f^\ell)^*\mathcal R^{(k)}_{T^*_{\Pbb^{m^\p}_\Cbb, -}(\ell);T^*_{\Pbb^m_\Cbb, -}(\ell)}\big) = (0)$ for all $\ell \leq \ell(k)$. Since the exterior algebra $\wedge^\bt T^*_{\Pbb^{m^\p}_\Cbb, -}$ is bounded, i.e., $\wedge^k T^*_{\Pbb^{m^\p}_\Cbb, -} = (0)$ for $k < 0$ and $k> \mathrm{rank}~T^*_{\Pbb^{m^\p}_\Cbb, -}$, there are only finitely many such values $\ell(k)$ to consider. Set $\ell_0 = \min\{\ell(k): 2\leq k \leq \mathrm{rank}~T^*_{\Pbb^{m^\p}_\Cbb, -}\}$. By Proposition \ref{kdmlkcnekjkjejnevne} and Theorem \ref{ldfjnvnjfjdkkskskdkd}, $f^\ell$ will be split for all $\ell \leq \ell_0$. 
\end{proof}

\section{Applications II: Projective Varieties}
\label{ofprijiovorvorvoirip}

\subsection{Projective Superspace}
Throughout this article, we have studied supermanifolds by reference to their model. The treatment so far is perhaps a little abstract so we consider more concrete examples presently. Complex affine superspace $\Abb_\Cbb^{m+1|n}$ is the superspace with global function algebra $\Cbb[x_0, \ldots, x_n|\q_1, \ldots, \q_m]$, the polynomial algebra defined by relations $x_ix_j = x_jx_i, \q_a\q_b = -\q_b\q_a$ and $x_i\q_a = \q_ax_i$. Complex superspace $\Cbb^{m+1|n}$ is the split model associated to the model $(\Cbb^{m+1}, \oplus^n\Oc)$, where $\Oc$ is the structure sheaf of $\Cbb^{m+1}$ and $\oplus^n\Oc$ is the $n$-fold direct sum. In the notation in this article, $\Cbb^{m+1|n} = e_{(\Cbb^{m+1}, \oplus^n\Oc)}$. A construction of projective superspace, as one might encounter in the literature, can be found in \cite{YMAN}. It proceeds along lines similar to the construction of projective space. For shorthand set $\Cbb[x|\q] = \Cbb[x_0, \ldots, x_n|\q_1, \ldots, \q_m]$. The multiplicative group $\mathbb G_m = \Cbb^\times$ acts on $\Cbb[x|\q]$ by scaling $x\mapsto \lam x$ and $\q\mapsto \lam\q$ for all $\lam\in \mathbb G_m$. In viewing $(x|\q)$ as a system of global coordinates on $\Cbb^{m+1|n}$ we see that $\mathbb G_m$ will act on $\Cbb^{m+1|n}$, leaving fixed the origin $(0|0)$. The quotient $\big(\Cbb^{m+1|n}- \{(0|0)\}\big)/\mathbb G_m$ is defined to be the projective superspace $\Pbb^{m|n}_\Cbb$. An instructive exercise is to verify the following, details of which we omit here. 

\begin{LEM}\label{rhfuhf8hf983j0fj30}
Projective superspace $\Pbb^{m|n}_\Cbb$ is the split model associated to the model $\big(\Pbb^m_\Cbb, \oplus^n\Oc_{\Pbb^m_\Cbb}(-1)\big)$. \qed
\end{LEM}

\subsection{Degree-$d$ Subvarieties}
Subvarieties of affine superspace $\Abb_\Cbb^{m+1|n}$ are described by prime ideals $\Ic$ in $\Cbb[x|\q]$. Those prime ideals which are homogeneous define subvarieties of $\Pbb_\Cbb^{m|n}$. Consider a homogeneous, prime ideal $\Ic\subset \Cbb[x|\q]$ generated by polynomials $\{P^\al(x|\q)\}_{\al\in I}$, where $I$ is a finite index set and
\begin{align*}
P^\al(x|\q)
&=
\sum_{|\mu|>0}P^{\al|\mu}(x)\q_\mu
\\
&=
P^{\al|0}(x) + P^{\al|a}(x)\q_a +  P^{\al|ij}(x)\q_i\q_j + \ldots
\end{align*}
where $\mu$ is a multi-index; $|\mu|$ is its length; and $P^{\al|\mu}(x)$ are polynomials in $x$ and all the free indices are implicitly summed. The set $\{P^{\al|0}(x)\}_{\al\in I}$ generates a homogeneous ideal in $\Cbb[x]$ and so defines a subvariety $V$ of $\Pbb^m_\Cbb$. Since the even and odd coordinates $x$ and $\q$ have the same degree under the scaling action of $\mathbb G_m = \Cbb^\times$, we see that $P^\al(\lam x|\lam \q) = \lam^dP^\al(x|\q)$ if and only if 
\begin{align}
\deg P^{\al|\mu}(x) = d - |\mu|
\label{rfbf8h3hf983j03j}
\end{align}
for all $\mu$. The subvariety $\Vc\subset \Pbb^{m|n}_k$ defined by $\Ic$ has degree $d$ if and only if $\Ic$ is generated by homogeneous polynomials $\{P^\al(x|\q)\}_{\al\in I}$ whose coefficients satisfy $\eqref{rfbf8h3hf983j03j}$ for all $\al, \mu$. Some general properties to observe are:
\begin{enumerate}[$(i)$]
	\item $\Vc_\red =V$;
	\item the odd conormal sheaf $T^*_{V, -}$ is defined as the cokernel of the syzygy generated by $\big(\sum P^{\al|a}(x)\q_a\big)_{\al\in I}$ in $\oplus^n \Oc_{\Pbb^m_\Cbb}(-1)$.
\end{enumerate}
The odd conormal sheaf of the embedding $(V, T_{V, -}^*)\subset (\Pbb^{m}_\Cbb, \oplus^n\Oc_{\Pbb^m_\Cbb}(-1))$ of models is generated by the relation $\sum P^{\al|a}(x)\q_a$ for all $\al\in I$. Hence if $P^{\al|a}(x) = 0$ for all $\al$ and $a$, this embedding of models will be even. And in this case $T_{V, -}^* = i^*\oplus^n\Oc_{\Pbb^m_\Cbb}(-1)$, for $i: V\subset \Pbb^m_\Cbb$ the holomorphic embedding of reduced spaces.

\begin{PROP}\label{rfh98fh98jf83j0j30}
Any linear subvariety $\Vc\subset \Pbb_\Cbb^{m|n}$ is split.
\end{PROP}

\noindent
By Lemma \ref{rhfuhf8hf983j0fj30}, the statement of Lemma \ref{rfh98fh98jf83j0j30} above is a particular instance of the more general statement in Example \ref{9008f9gf73f}. In Example \ref{9008f9gf73f} however it was necessary reference a classical result of Van de Ven on the nature of linear subvarieties of projective spaces. We will present below a simple argument which does not require Van de Ven's result. 
\\\\
\emph{Proof of Proposition $\ref{rfh98fh98jf83j0j30}$.}
Let $\Vc\subset \Pbb^{m|n}_\Cbb$ be given by the ideal $\Ic$ and generators $F = \big(\{P^{\al}\}_{\al\in I}\big)$. If $\Vc$ is linear, $d = 1$. Then from \eqref{rfbf8h3hf983j03j} we see that $\deg P^{\al|\mu} < 0$ for $|\mu|> 1$, which means $P^{\al|\mu} = 0$ for $|\mu|> 1$. Hence the minimal non-splitting degree of $F$ is $\8$, i.e., $m_{F^\pm} =\8$ (c.f., \eqref{rhf4gf78hf983jf389fh39hf3}). In using that $\Pbb_\Cbb^{m|n}$ is split by Lemma \ref{rhfuhf8hf983j0fj30}, this proposition follows from Lemma \ref{krknjviurvburnvien}. 
\qed

\subsection{The Rational Normal Curve}
Consider a subvariety $V\subset \Pbb_\Cbb^m$ of degree $d$, defined by a homogeneous, prime ideal $\big(\{P^{\al|0}(x)\}_{\al\in I}\big)$. To any $\lam\in \Cbb$ consider the ideal $\Ic_\lam$ generated by 
\begin{align}
P^\al(x|\q) = P^{\al|0}(x) + \lam \q_1\cdots\q_d. 
\label{rhf49hf98f83j093jf903jf0}
\end{align}
Then $\Ic_\lam$ will define a degree-$d$ subvariety $\Vc_\lam\subset \Pbb_\Cbb^{m|d}$, with $(\Vc_\lam)_\red = V$. A natural question to ask is whether $\Vc_\lam$ so described is split or not. We address this in the case where $V\subset \Pbb_\Cbb^d$ is the rational normal curve, i.e., a degree-$d$ embedding of $\Pbb_\Cbb^1$. 

\begin{THM}\label{orohfuirgfuirhf98h49}
For $\lam\in \Cbb$, let $\Vc_\lam\subset \Pbb_\Cbb^{d|d}$ be given by $F = P(x) + \lam \q_1\cdots \q_d$, where $(V_\lam)_\red\subset \Pbb_\Cbb^d$, defined by $(P(x))$, is the rational normal curve of degree $d$. Then:
\begin{enumerate}[(i)]
	\item if $d = 2$, $\Vc_\lam$ will be split if and only if $\lam = 0$;
	\item if $d\neq 2$, $\Vc_\lam$ is split for any $\lam$.
\end{enumerate}
\end{THM} 

\begin{REM}
\emph{Theorem \ref{orohfuirgfuirhf98h49}\emph{(i)} was also addressed by Onishchik and Bunegina in \cite{ONIBUG}. There the authors argued, by reference to transition data, that the superspace quadric in $\Pbb^{2|2}_\Cbb$ is non-split. We recover these results in Theorem \ref{orohfuirgfuirhf98h49}\emph{(i)} without recourse to transition data.}
\end{REM}

\begin{proof}
Let $\Ic_\lam$ be the ideal defining $\Vc_\lam$. The case $d= 1$ is addressed in Proposition \ref{rfh98fh98jf83j0j30}. For $d>1$, observe that the embedding of models $(V, T^*_{V, -})\subset (\Pbb^d_\Cbb, \oplus^d \Oc_{\Pbb^d_\Cbb}(-1))$ will be even. Since $i: V = (\Vc_\lam)_\red\subset \Pbb^d_\Cbb$ is the degree-$d$ embedding of the rational normal curve, $V$ is isomorphic to $\Pbb^1_\Cbb$ and $T^*_{V, -} \cong i^*\oplus^d\Oc_{\Pbb^d_\Cbb}(-1) = \oplus^d\Oc_{\Pbb^1_\Cbb}(-d)$. The model for the degree-$d$ subvariety $\Vc_\lam$ is therefore $(\Pbb^1_\Cbb, \oplus^d\Oc_{\Pbb^1_\Cbb}(-d))$. We will now focus on part \emph{(i)} and so set $d = 2$. The model for $\Vc_\lam$ in degree $2$ is $(\Pbb^1_\Cbb, \oplus^2\Oc_{\Pbb^1_\Cbb}(-2))$. Since $\Pbb_\Cbb^{2|2}$ is split we can use Theorem \ref{h89h890309i390} to evaluate the obstructions to splitting. We firstly recall some classical theory: 
\begin{enumerate}[$\bt$]
	\item any holomorphic vector bundle on $\Pbb_\Cbb^1$ will split into a sum of holomorphic line bundles (see \cite[p. 12]{OSS});
	\item a rank $r$ vector bundle $E$ on a rational curve is said to be \emph{balanced} if $\Ec \cong \Oc_{\Pbb_\Cbb^1}(k)^{\oplus s} \oplus \Oc_{\Pbb_\Cbb^1}(k-1)^{\oplus r-s}$, where $\Ec$ is the sheaf of holomorphic sections of $E$;
	\item for a rational curve $C\subset \Pbb_\Cbb^n$, the restriction $\wedge^\ell T_{\Pbb_\Cbb^n}|_C$ is balanced for all $\ell$. This a consequence of the Grauert-M\"ulich theorem (see \cite[p. 104]{OSS}). 
\end{enumerate}
Now $T_{\Pbb_\Cbb^2}$ is a holomorphic vector bundle on $\Pbb_\Cbb^2$ of degree $3$. Hence its restriction to $V$ will have degree $6$. From the above facts, this leads therefore to:
\[
T_{\Pbb_\Cbb^2}|_{V} \cong \Oc_{\Pbb_\Cbb^1}(3)\oplus \Oc_{\Pbb_\Cbb^1}(3).
\]
In using now that $T_{V} = \Oc_{\Pbb_\Cbb^1}(2)$ and $\nu_{V/\Pbb_\Cbb^2} = \Oc_{\Pbb_\Cbb^1}(4)$ the normal bundle sequence to the embedding $V\subset \Pbb_\Cbb^2$ is:
\begin{align}
0 \lra \Oc_{\Pbb_\Cbb^1}(2) \lra  \Oc_{\Pbb_\Cbb^1}(3)\oplus \Oc_{\Pbb_\Cbb^1}(3)\lra  \Oc_{\Pbb_\Cbb^1}(4)\lra0.
\label{rfg784hf9hf83j0fj390fj093fj03}
\end{align}
With $T^*_{V, -} = \Oc_{\Pbb_\Cbb^1}(-2)^{\oplus 2}$, dualising \eqref{rfg784hf9hf83j0fj390fj093fj03} and applying $\mathcal Hom_{\Oc_{V_\red}}\big(-, \wedge^2T_{V, -}^*\big)$, the induced sequence on cohomology gives,
\[
0 \lra \Cbb \stackrel{\dt}{\lra} \Cbb \lra 0.
\]
The boundary map $\dt$ above can be identified with $\dt_2$ in \eqref{cfgjgjgjgjfjjfhfgbgjf} since the embedding $\Vc_\lam\subset \Pbb^{d|d}_\Cbb$ is even. Note in particular that $\dt$ is an isomorphism. Now recall that $\Vc_\lam$ is defined by the pair $(\Ic, F)$ where $F = P(x) + \lam \q_1\q_2$. The term $h_F = \lam\q_1\q_2$ pulls back to a global homomorphism in $\Hom_{\Oc_V}\big(\nu^*_{V/\Pbb_\Cbb^2}, \wedge^2T^*_{V, -}\big)\cong \Cbb$ and so can be identified with $\lam$. By Theorem \ref{rf8hg984h9g8gj4j0} we see that $\dt(\lam)$ will be the obstruction class of $\Vc_\lam$. As we have observed, $\dt$ is an isomorphism so therefore if $\lam\neq 0$, $\Vc_\lam$ will be non-split. This settles part \emph{(i)}. Regarding \emph{(ii)}, note that the result holds trivially from Theorem \ref{ldfjnvnjfjdkkskskdkd} when $d$ is odd, since $\nu^*_{V/\Pbb_\Cbb^d, -} = (0)$. For $d$ even, we firstly recall the classical fact:
\begin{enumerate}[$\bt$]
	\item let $C\subset \Pbb_\Cbb^d$ be a rational curve of degree-$d$. Then $\nu_{C/\Pbb_\Cbb^d} \cong \Oc_{\Pbb_\Cbb^1}(d+2)^{\oplus d-1}$. 
\end{enumerate}
Hence,
\[
\mathcal Hom_{\Oc_{V}}\big(\nu^*_{V/\Pbb_\Cbb^d}, \wedge^dT_{V, -}^*\big)\cong \Oc_{\Pbb_\Cbb^1} (d + 2 - d^2)^{\oplus (d-1)}. 
\]
Since $d^2 > d + 2$ for $d> 2$, the above sheaf cannot have any global sections. Part \emph{(ii)} now follows from Theorem \ref{ldfjnvnjfjdkkskskdkd}. 
\end{proof}

\section*{Concluding Remarks}

\noindent
The obstruction classes to splitting supermanifolds appear prominently throughout this article. They are indispensable to the understanding of complex supermanifolds but their application in theoretical physics remains unclear. More recently, they have been considered in the framework of superstring theory, arising there as the impediment to the calculation of the superstring scattering amplitude to loop orders greater than two. We propose another potential application of the obstruction classes which would be interesting to pursue in future work. 
\\\\
In the paper by Sethi in \cite{SETHI} and Aganagic and Vafa in \cite{AGAVAFA}, certain superspace quadrics are proposed as mirrors for the rigid K\"ahler manifolds appearing in Landau-Ginzberg sigma models. Based on Theorem \ref{orohfuirgfuirhf98h49}\emph{(i)}, we might expect these superspace quadrics to be non-split and indeed this is what is deduced in a subsequent paper \cite{BETTQUAD}. In which case, the mirror map described by Sethi and Aganagic-Vafa ought to exchange the K\"ahler parameter with  obstruction classes to splitting the mirror superspace quadric. This could lead to interesting interplay between K\"ahler geometry and complex supergeometry and so would be an interesting line of research to pursue.

\appendix
\numberwithin{equation}{section}

\section{Proof of Lemma \ref{rbf8gf87h4f98h389fj0j30}}
\label{rf67f47gf87hf983hf03j0}

\noindent 
Let $\Xfr = (X, \Oc_\Xfr)$ be a supermanifold. We wish to show: if the following exact sequences of sheaves on $X$:
\begin{align}
0 \ra \Jc_\Xfr\ra \Oc_\Xfr \ra \Oc_X\ra0
\label{fiorjvropjvprpvrp}
\end{align}
and 
\begin{align}
0\ra \Jc^2_{T^*_{X, -}} \ra \Jc_{T^*_{X, -}} \ra T^*_{X, -}\ra0
\label{fjgbnbnngthrrorff}
\end{align}
are split exact, then $\Xfr$ is split. That is, there exists a global isomorphism $\Oc_\Xfr \stackrel{\sim}{\ra}\wedge^\bt T^*_{X, -}$. A splitting of \eqref{fiorjvropjvprpvrp} shows that $\Oc_\Xfr$ will admit the structure of an $\Oc_X$-algebra. Then with a splitting of \eqref{fjgbnbnngthrrorff}, we are in the situation where: we have a map $j: T^*_{X, -}\ra \Oc_\Xfr$ of an $\Oc_X$-module into the unital $\Oc_X$-algebra $\Oc_\Xfr$ such that $j(\xi)^2 = 0$, by supercommutativity of $\Oc_\Xfr$. Hence by the universal property of exterior algebras, there will exist a unique algebra morphism $\psi : \wedge^\bt T^*_{X, -}\ra \Oc_\Xfr$ such that the following diagram commutes:
\begin{align}
\xymatrix{
T^*_{X, -} \ar[dr]_j \ar[rr] & & \wedge^\bt T^*_{X, -}\ar[dl]^\psi
\\
& \Oc_\Xfr & 
}
\label{h98rhf89j4f98j0}
\end{align}
It remains to show that $\psi$ is an isomorphism. To see this, consider a cover $(U_i)$ for $X$ and local splittings $\vp_i : \wedge^\bt T^*_{X, -}(U_i)\stackrel{\sim}{\ra} \Oc_\Xfr(U_i)$ which exist since $\Xfr$ is a supermanifold. As $\psi$ comes from the universal property, $\psi|_{U_i} :  \wedge^\bt T^*_{X, -}(U_i)\ra \Oc_\Xfr(U_i)$ will inherit this property, i.e., it will be unique so there will exist a morphism $\phi_i : \wedge^\bt T^*_{X, -}(U_i)\ra\wedge^\bt T^*_{X, -}(U_i)$ commuting the following diagram:
\begin{align}
\xymatrix{
\ar[dr]_{\vp_i} \wedge^\bt T^*_{X, -}(U_i) \ar[rr]^{\phi_i} & & \wedge^\bt T^*_{X, -}(U_i)\ar[dl]^{\psi|_{U_i}}
\\
& \Oc_\Xfr(U_i)
}
\label{kmdklnuebieoe}
\end{align}
Now $\vp_i$ is an isomorphism and commutativity of \eqref{kmdklnuebieoe} gives $\vp_i = \psi|_{U_i}\circ \phi_i$. Hence both $\psi|_{U_i}$ and $\phi_i$ must be injective. Since $\phi_i$ is an injective homomorphism of sheaves, it will be injective on stalks. Then at the level of stalks, note that $\phi_i$ will will be an injective morphism of a finite rank algebra into itself, so it will therefore be an isomorphism. Hence $\phi_i$ is an isomorphism of sheaves. From \eqref{kmdklnuebieoe} now, we see that $\psi|_{U_i}$ can be written as a composition of isomorphisms. Therefore $\psi|_{U_i}$ is an isomorphism for each $U_i$. It follows that $\psi$ itself is an isomorphism of sheaves.\qed

\bibliographystyle{alpha}
\bibliography{Bibliography}

\hfill
\\
\noindent
\small
\textsc{
Kowshik Bettadapura 
\\
\emph{Yau Mathematical Sciences Center} 
\\
Tsinghua University
\\
Beijing, 100084, China}
\\
\emph{E-mail address:} \href{mailto:kowshik@mail.tsinghua.edu.cn}{kowshik@mail.tsinghua.edu.cn}

\end{document}